\documentclass[12pt, a4paper, abstracton]{scrartcl}
\pdfoutput=1

\usepackage{amsmath, amsthm, amsfonts, amssymb}
\usepackage[utf8]{inputenc}
\usepackage[nottoc]{tocbibind} 
\usepackage{color}
\usepackage{slashed} 
\usepackage[all, cmtip]{xy}
\usepackage{hyperref} 
\usepackage{chngcntr} 
\usepackage{graphicx} 
\usepackage[english]{babel}
\usepackage{microtype}
\usepackage{bm} 
\usepackage{xfrac} 
\usepackage{mathtools} 

\newcommand{\C}{C_{-\infty}^\ast}
\newcommand{\IN}{\mathbb{N}}
\newcommand{\IZ}{\mathbb{Z}}

\newcommand{\IR}{\mathbb{R}}
\newcommand{\IC}{\mathbb{C}}

\newcommand{\injrad}{\operatorname{inj-rad}}
\newcommand{\Rm}{\operatorname{Rm}}

\newcommand{\IB}{\mathfrak{B}}

\newcommand{\Tr}{\mathfrak{Tr}}

\newcommand{\supp}{\operatorname{supp}}

\newcommand{\id}{\operatorname{id}}

\newcommand{\card}{\#}

\newcommand{\trace}{\operatorname{tr}}
\newcommand{\ch}{\operatorname{ch}}
\newcommand{\vol}{\operatorname{vol}}
\newcommand{\ind}{\operatorname{ind}}

\newcommand{\op}{\mathrm{op}}

\newcommand{\Cpoln}{C^\ast_{\mathrm{pol},n}}
\newcommand{\Cpol}{C^\ast_{\mathrm{pol}}}

\newcommand{\HCocont}{H \! C^\mathrm{cont}}

\newcommand{\HPucont}{H \! P_\mathrm{cont}}
\newcommand{\PHCocont}{P \! H \! C^\mathrm{cont}}
\newcommand{\PHC}{P \! H \! C}

\newcommand{\HC}{H \! C}
\newcommand{\HP}{H \! P}

\newcommand{\Cl}{C^{\lambda}}
\newcommand{\Huf}{H^\mathrm{uf}}
\newcommand{\Hulf}{H^\mathrm{ulf}}
\newcommand{\Hufpol}{H^{\mathrm{pol}}}

\newcommand{\Cuf}{C^\mathrm{uf}}
\newcommand{\Culf}{C^\mathrm{ulf}}
\newcommand{\Cufpol}{C^{\mathrm{pol}}}

\newcommand{\uf}{\mathrm{uf}}

\newcommand{\HbdR}{H_{b, \mathrm{dR}}}
\newcommand{\HcdR}{H_{c, \mathrm{dR}}}

\newcommand{\length}{\mathrm{length}}
\newcommand{\HR}{H \! R}
\newcommand{\CR}{C \! R}

\newcommand{\HX}{H \! X}
\newcommand{\CX}{C \! X}
\newcommand{\uRoe}{\IC_{-\infty}^\ast}
\newcommand{\Frechet}{Fr\'{e}chet }
\newcommand{\Poincare}{Poincar\'{e} }
\newcommand{\Folner}{F{\o}lner }
\newcommand{\Spakula}{\v{S}pakula }

\newcommand{\alg}{\mathrm{alg}}

\newcommand{\interior}{\mathrm{int}}
\newcommand{\WDpol}{W^{\infty,1}_{\Delta,\mathrm{pol}}}
\newcommand{\dist}{\operatorname{dist}}

\newcommand*{\largecdot}{\raisebox{-0.25ex}{\scalebox{1.2}{$\cdot$}}}

\DeclareMathOperator{\hatotimes}{\hat{\otimes}}
\DeclareMathOperator{\barotimes}{\bar{\otimes}}

\newtheorem{thm}{Theorem}[section]
\newtheorem*{thm*}{Theorem}
\newtheorem*{thmA*}{Theorem A}
\newtheorem*{thmB*}{Theorem B}
\newtheorem*{applC*}{Application C}
\newtheorem*{applD*}{Application D}
\newtheorem*{applE*}{Application E}
\newtheorem*{thmF*}{Theorem F}
\newtheorem*{corG*}{Corollary G}
\newtheorem*{thmH*}{Theorem H}

\newtheorem*{roesthm*}{Roe's Index Theorem}

\newtheorem*{atiyahsingerthm*}{Atiyah--Singer Index Theorem}

\newtheorem{cor}[thm]{Corollary}
\newtheorem*{cor*}{Corollary}
\newtheorem{lem}[thm]{Lemma}
\newtheorem{prop}[thm]{Proposition}

\newtheorem{question}[thm]{Question}
\theoremstyle{definition}
\newtheorem{rem}[thm]{Remark}
\newtheorem*{rem*}{Remark}
\newtheorem{example}[thm]{Example}
\newtheorem{examples}[thm]{Examples}
\newtheorem*{example*}{Example}
\newtheorem*{examples*}{Examples}
\newtheorem{defn}[thm]{Definition}

\numberwithin{equation}{section}

\counterwithout{footnote}{section}

\makeatletter
\def\blfootnote{\gdef\@thefnmark{}\@footnotetext}
\makeatother

\begin{document}

\title{Rough index theory on spaces of\\$\mathclap{\mbox{polynomial growth and contractibility}}$}
\author{Alexander Engel}
\date{}
\maketitle

\vspace*{-3.5\baselineskip}
\begin{center}
\footnotesize{
\textit{
Fakult{\"a}t f{\"u}r Mathematik\\
Universit{\"a}t Regensburg\\
93040 Regensburg, GERMANY\\
\href{mailto:alexander.engel@mathematik.uni-regensburg.de}{alexander.engel@mathematik.uni-regensburg.de}
}}
\end{center}

\vspace*{0.5\baselineskip}
\begin{abstract}
We will show that for a polynomially contractible manifold of bounded geometry and of polynomial volume growth every coarse and rough cohomology class pairs continuously with the $K$-theory of the uniform Roe algebra.

As an application we will discuss non-vanishing of rough index classes of Dirac operators over such manifolds, and we will furthermore get higher-codimensional index obstructions to metrics of positive scalar curvature on closed manifolds with virtually nilpotent fundamental groups.

We will give a computation of the homology of (a dense, smooth subalgebra of) the uniform Roe algebra of manifolds of polynomial volume growth.

\blfootnote{\textit{$2010$ Mathematics Subject Classification.} Primary:\ 58J22; Secondary:\ 46L80,\ 19K56.}
\blfootnote{\textit{Keywords and phrases.} Novikov conjecture, uniform Roe algebra, uniformly finite homology.}
\end{abstract}

\tableofcontents

\section{Introduction}

J.~Roe \cite{roe_coarse_cohomology} invented coarse cohomology which pairs with the $K$-theory of the algebraic Roe algebra and showed a corresponding index theorem. The main Novikov-type question is then whether this pairing is continuous so that it extends to a pairing with the $K$-theory of the completed Roe algebra. Degree zero classes always pair continuously and J.~Roe also characterized the degree one classes that do this. In the degree zero case this leads (in the uniform setting) to his index theorem for amenable spaces \cite{roe_index_1} and in the degree one case to his partitioned manifolds index theorem \cite{roe_partitioning}.

The first goal of this paper is to prove the following theorem. Note that we work in the rough category and not in the coarse category as J.~Roe does, i.e., we pair not with the usual Roe algebra but with the uniform one. This enables us to treat not only coarse cohomology, but also rough cohomology, which leads to the fact that we may also consider, e.g., pairings incorporating so-called \Folner exhaustions.

\begin{thmA*}
Let $M$ be a polynomially $k$-connected manifold of bounded geometry having polynomial volume growth.

Then for all $q \le k$ every element in the coarse cohomology $\HX^q(M)$ and in the rough cohomology $\HR^q(M)$ pairs continuously with $K_\ast(C_u^\ast(M))$.
\end{thmA*}

In the above theorem $C_u^\ast(M)$ denotes the uniform Roe algebra of $M$ and polynomial $k$-connectedness means the following: there is a polynomial $P$ such that for all $i \le k$ every $L$-Lipschitz map $S^i \to M$ can be contracted by a $P(L)$-Lipschitz contraction. Examples are universal covers (equipped with the pull-back metric) of closed manifolds $K$ with virtually nilpotent fundamental group and $\pi_i(K) = 0$ for $2 \le i \le k$.

The top-dimensional case of the above theorem was treated by G.~Yu \cite[Section 4]{yu_cyclic} (for manifolds with subexponential growth and a subexponential rate of contractibility). So the above theorem can be regarded as a generalization of G.~Yu's result to encompass coarse, resp.~rough classes not in the top-dimension.

Combining results of J.~Roe (\cite[Sections 5.4 \& 6.6]{roe_coarse_cohomology} and \cite{roe_hyperbolic}) we conclude that he proved a somewhat orthogonal result to the above theorem: on hyperbolic manifolds coarse cohomology classes always pair against the $K$-theory of the Roe algebra.

The proof of Theorem A is based on the following diagram:
\begin{equation}
\label{eq:main_diag}
\xymatrix{K_\ast^u(M) \ar[rrr]^{\mu^u_\ast} \ar[ddd] & & & K_\ast(C^\ast_u(M)) \ar@{=}[d]\\
& & K^{\mathrm{alg}}_\ast(\IC_{-\infty}^\ast(M)) \ar[d]_{\ch_\ast} \ar[r] & K_\ast(\Cpol(M)) \ar[d]\\
& & \PHC_\ast(\IC_{-\infty}^\ast(M)) \ar[d]_{\chi_\ast} \ar[r] & \PHCocont_\ast(\Cpol(M)) \ar[d]\\
H^\infty_\ast(M) \ar[rr] & & \Huf_\ast(Y) \ar[r] & \Hufpol_\ast(Y)}
\end{equation}

The top horizontal map $K_\ast^u(M) \to K_\ast(C_u^\ast(M))$ is the rough Baum--Connes assembly map defined by \Spakula \cite[Section 9]{spakula_uniform_k_homology}, where $K_\ast^u(M)$ is the uniform $K$-homology of~$M$ (also defined by \Spakula).

$\Cpol(M)$ is a certain dense and smooth subalgebra of $C_u^\ast(M)$, i.e., the $K$-theories of these two algebras coincide. Section \ref{sec:dense_smooth_subalgebras} is dedicated to the definition of $\Cpol(M)$ and to the proof that it is closed under holomorphic functional calculus.

The map $K_\ast^u(M) \to H_\ast^\infty(M)$ is the uniform homological Chern character from uniform $K$-homology into $L^\infty$-simplicial homology of $M$ with complex coefficients\footnote{We triangulate $M$ as a simplicial complex of bounded geometry, which is possible since $M$ has bounded geometry; see \cite[Theorem 2.1]{attie}.} which the author defined \cite{engel_indices_UPDO}, the map $H_\ast^\infty(M) \to \Huf_\ast(Y)$ is the natural map of $L^\infty$-homology into uniformly finite homology of Block and Weinberger, where $Y \subset M$ is a discretization of $M$, and $\Hufpol_\ast(Y)$ will be defined in Section \ref{sec:defn_ufpol_hom} by completing the chain complex of uniformly finite chains in a certain \Frechet topology.

The map $\ch_\ast$ is the usual Chern character from $K$-theory to periodic cyclic homology, and the character map $\chi_\ast$ will be constructed in Section \ref{sec:rough_character}. The two vertical lower right maps are continuous extensions of $\ch_\ast$, resp.~of $\chi_\ast$.

Note that G.~Yu's arguments in \cite{yu_cyclic} and also \cite{yu_zero} are similar to ours in the sense that he constructs there also certain Chern characters.

From \cite{engel_indices_UPDO} it follows that $K_\ast^u(M) \to H^\infty_\ast(M)$ is rationally injective, and it is known that if $M$ is contractible in a uniform sense then $H^\infty_\ast(M) \to \Huf_\ast(Y)$ is an isomorphism (Lemmas~\ref{lemh23wefwe} \& \ref{lemui2323}). So the rough Novikov conjecture (i.e., rational injectivity of the rough assembly map $\mu^u_\ast$) follows from injectivity of $\Huf_\ast(Y) \to \Hufpol_\ast(Y)$.

We will show in Theorem \ref{thm:bounded_below} that under the assumptions of our main theorem the lower horizontal composition $H^\infty_q(M) \to \Hufpol_q(Y)$ will be bounded from below for all $q \le k$, where we use on the space $H^\infty_q(M)$ the $L^\infty$-seminorm and on the space $\Hufpol_q(Y)$ the seminorms from Definition \ref{defn_pol_chains}. In Corollary~\ref{coriso} this will be improved to the fact that the composition $H^\infty_q(M) \to \Hufpol_q(Y)$ is even a topological isomorphism. Let us state this as one of our main theorems:

\begin{thmB*}
Let $M$ be a polynomially $k$-connected manifold of bounded geometry having polynomial volume growth.

Then for every $q \le k$ both of the maps $H_q^\infty(M) \to \Huf_q(Y)$ and $\Huf_q(Y) \to \Hufpol_q(Y)$ are topological isomorphisms.
\end{thmB*}

Index theory enters with \cite{engel_indices_UPDO}: if $P$ is a symmetric and elliptic uniform pseudodifferential operator, then the image of $[P] \in K_\ast^u(M)$ in $H_\ast^\infty(M)$ under the uniform homological Chern character is the \Poincare dual of the cohomological index class $\ind(P) \in \HbdR^\ast(M)$ of $P$, where the latter denotes bounded de Rham cohomology. Note that the special case of Dirac operators is essentially due to J.~Roe \cite{roe_coarse_cohomology} and originally this kind of index theorems go back to Connes--Moscovici \cite{connes_moscovici}.

Theorem A only becomes interesting if we can find interesting coarse cohomology classes to pair with. So let us quickly discuss two examples.

\begin{applC*}
Let $M$ be a Riemannian manifold of bounded geometry.

If $M$ is polynomially contractible (i.e., polynomially $k$-connected for every $k \in \IN$) and has polynomial growth and if $D$ is a Dirac operator over $M$, then $\mu^u_\ast [D] \not= 0 \in K_\ast(C_u^\ast M)$.
\end{applC*}

Since we assume in the above application polynomial contractibility for all $k \in \IN$, this particular conclusion of non-vanishing of $\mu^u_\ast [D]$ already follows from the above cited work of G.~Yu \cite{yu_cyclic}. The conclusion also follows from Gong--Yu \cite[Theorem 4.1]{gong_yu}, since they showed the coarse Baum--Connes conjecture for spaces with subexponential growth.

Application C implies, e.g., a non-existence theorem for metrics of uniform positive scalar curvature:

\begin{applD*}
Let $M$ be a polynomially contractible manifold of bounded geometry and of polynomial volume growth.

Then $M$ does not admit a metric of uniformly positive scalar curvature in its strict quasi-isometry class.
\end{applD*}

Hanke--Pape--Schick \cite{hanke_pape_schick} proved the following theorem: if $M$ is a closed, connected spin manifold with $\pi_2(M) = 0$, and $N \subset M$ is a codimension two submanifold with trivial normal bundle and such that the induced map $\pi_1(N) \to \pi_1(M)$ is injective, then $\alpha(N) \in K_\ast(C^\ast \pi_1(N))$ is an obstruction against the existence of a psc-metric on $M$.

Using Theorem A we will generalize this to higher codimensions, but we will only have the weaker obstruction $\hat{A}(N)$ instead of $\alpha(N)$. We will also have the restriction that $\pi_1(M)$ has to be virtually nilpotent due to the restrictions of our technique, i.e., we need that the universal cover of $M$ satisfies the assumptions of Theorem A.

Note that the case that the codimension of $N$ equals the dimension of $M$ is allowed in our next theorem and is nothing else but the well-known conjecture that no closed aspherical manifold admits a psc-metric.

Refined versions of the following result were obtained by the author some time after this paper was first posted on the arXiv \cite{engel_wrongway}. But the technique used to prove the refined results is different from the one used in this paper.

\begin{applE*}
Let $M$ be a closed, connected manifold with $\pi_1(M)$ virtually nilpotent and $\pi_i(M) = 0$ for $2 \le i \le q$. Assume moreover that $N \subset M$ is a connected submanifold of codimension $q$ and with trivial normal bundle.

If $\hat A(N) \not= 0$ then $M$ does not admit a metric of positive scalar curvature.
\end{applE*}

Our study of Diagram~\eqref{eq:main_diag} goes on by proving in Section~\ref{seciu234e} the following result:

\begin{thmF*}
Let $M$ be a manifold of bounded geometry with polynomial volume growth, and let $Y \subset M$ be a discretization.

Then the character map $\chi$ induces an isomorphism
\[\chi_\ast \colon \PHCocont_\ast(\Cpol(M)) \xrightarrow{\cong} \Hufpol_\ast(Y),\]
where $\ast$ is either $\ast = 0$ on $\PHC$-theory and $\ast = \text{even}$ on homology, or it is $\ast = 1$ on $\PHC$-theory and $\ast = \text{odd}$ on homology.
\end{thmF*}

The above theorem has a lot of theoretical value. Since the coarse index pairings factor through the character map $\chi_\ast$, the injectivity statement in Theorem F means that (at least in the case of polynomial growth) no information is lost by $\chi_\ast$ (since an element in the kernel of $\chi_\ast$ would be non-detectable by J.~Roe's coarse index theory).

Surjectivity of $\chi_\ast$ corresponds to the idea that there exist no superfluous elements in $\Hufpol_\ast(Y)$ which we might detect by pairing with coarse cohomology classes but which do not have index theoretic value since they do not come from operators.

Let us mention a corollary of the combination of Theorem B and how Theorem F is proved (namely by constructing an inverse map which preserves propagation of cycles).

\begin{corG*}
Let $M$ be a polynomially contractible manifold of bounded geometry and with polynomial volume growth.

Then there is a constant $C < \infty$ such that every class in $\PHCocont_\ast(\Cpol(M))$ can be represented by a cycle consisting of operators of propagation at most $C$.
\end{corG*}

Gong--Yu \cite[Theorem 4.1]{gong_yu} proved the coarse Baum--Connes conjecture for manifolds of bounded geometry and with subexponential volume growth by showing that such manifolds are coarsely embeddable. It is reasonable to expect that in this case one can also prove the rough Baum--Connes conjecture, though there is currently no account of this in the literature. It follows that the Chern character $\ch_\ast\colon K_\ast(\Cpol(M)) \to \PHCocont_\ast(\Cpol(M))$ induces an isomorphism after taking the completed tensor product with $\IC$,\footnote{The reference for the completed topological tensor product $\largecdot \barotimes \IC$ and the reason why we must use it here may be found in \cite[Section 5.3]{engel_indices_UPDO}.} since all other outer arrows in Diagram~\eqref{eq:main_diag} are isomorphisms (in the case of the uniform homological Chern character $K_\ast^u(M) \to H_\ast^\infty(M)$ after taking the completed tensor product with $\IC$).

\begin{thmH*}
Let $M$ be a polynomially contractible manifold of bounded geometry and of polynomial volume growth. Assume that the rough Baum--Connes conjecture holds for $M$.

Then we have an isomorphism
\[\ch_\ast\colon K_\ast(C_u^\ast(M)) \barotimes \IC \xrightarrow{\cong} \PHCocont_\ast(\Cpol(M)).\]
\end{thmH*}

The rough Baum--Connes conjecture already provides a computation of the $K$-theory of the uniform Roe algebra. But in the above result the isomorphism is provided by mapping out of the $K$-theory of the uniform Roe algebra, whereas in the rough Baum--Connes conjecture the assembly map maps into the $K$-theory of the uniform Roe algebra and it is usually hard to invert the assembly map. So the above result is better suited to study particular elements of $K_\ast(C_u^\ast(M))$.

\paragraph{Acknowledgements}

I would like to thank Bernhard Hanke, Nigel Higson and Thomas Schick for valuable discussions with respect to this paper, and I would also like to thank Bernd Ammann, Ulrich Bunke, Clara Löh and Matthias Nagel for helpful discussions surrounding the ideas of this article and the techniques used in it. Furthermore, I would like to thank John Roe who provided to me a copy of the Ph.D.\ thesis of his student Boris Mavra. I thank Markus Land, Malte Pieper and Timm von Puttkamer from Bonn with whom I had helpful discussions during the PIMS Symposium on Geometry and Topology of Manifolds 2015, and Guillermo Corti{\~{n}}as, Joachim Cuntz and Tim Riley for answering questions related to this paper. Finally, I thank the anonymous referee for his or her comments.

I acknowledge the financial support from SFB 1085 ``Higher Invariants'' at the University of Regensburg, and from the Research Fellowship EN 1163/1-1 ``Mapping Analysis to Homology'', which are both financed by the DFG (Deutsche Forschungsgemeinschaft).

\section{Smooth subalgebras of the uniform Roe algebra}

\subsection{Manifolds of bounded geometry}

\begin{defn}
We will say that a Riemannian manifold $M$ has \emph{bounded geometry}, if
\begin{itemize}
\item the curvature tensor and all its derivatives are bounded, i.e., $\| \nabla^k \Rm (x) \| < C_k$ for all $x \in M$ and $k \in \IN_0$, and
\item the injectivity radius is uniformly positive, i.e., $\injrad_M(x) > \varepsilon > 0$ for all $x \in M$.
\end{itemize}
If $E \to M$ is a vector bundle with a metric and compatible connection, we say that $E$ has \emph{bounded geometry}, if the curvature tensor of $E$ and all its derivatives are bounded.
\qed
\end{defn}

We will now give the definition of uniform $C^\infty$-spaces together with a local characterization on manifolds of bounded geometry. The interested reader is refered to, e.g., \cite[Section 2]{roe_index_1} or \cite[Appendix A1.1]{shubin} for more information.

\begin{defn}[$C^r$-boundedness]
Let $e \in C^\infty(E)$. We will say that $e$ is \emph{$C^r$-bounded}, if $\| \nabla^i e \|_\infty < C_i$ for all $0 \le i \le r$.
\qed
\end{defn}

If $E$ and $M$ both have bounded geometry, being $C^r$-bounded is equivalent to the statement that in every normal coordinate chart and with respect to synchronous framings we have $|\partial^\alpha e_i(y)| < C_\alpha$ for every multiindex $\alpha$ with $|\alpha| \le r$ (where the constants $C_\alpha$ are independent of the chart and framing).

\begin{defn}[Uniform $C^\infty$-spaces]
\label{defn:uniform_frechet_spaces}
Let $E$ be a vector bundle of bounded geometry over $M$. We will denote the \emph{uniform $C^r$-space} of all $C^r$-bounded sections of $E$ by $C_b^r(E)$.

Furthermore, we define the \emph{uniform $C^\infty$-space $C_b^\infty(E)$}
\[C_b^\infty(E) := \bigcap_r C_b^r(E),\]
which is a \Frechet space.
\qed
\end{defn}

Now we get to Sobolev spaces on manifolds of bounded geometry. Much of the following material is again from \cite[Section 2]{roe_index_1} and \cite[Appendix A1.1]{shubin}.

Let $s \in C^\infty_c(E)$ for some vector bundle $E \to M$ with metric and connection $\nabla$. For $k \in \IN_0$ we define the global $H^k$-Sobolev norm of $s$ by
\begin{equation}\label{eq:sobolev_norm}
\|s\|_{H^k}^2 := \sum_{i=0}^k \int_M \|\nabla^i s(x)\|^2 dx.
\end{equation}

\begin{defn}[Sobolev spaces $H^k(E)$]\label{defn:sobolev_spaces}
Let $E$ be a vector bundle which is equipped with a metric and a connection. The \emph{$H^k$-Sobolev space of $E$} is the completion of $C^\infty_c(E)$ in the norm $\|\largecdot\|_{H^k}$ and will be denoted by $H^k(E)$.
\qed
\end{defn}

If $E$ and $M^m$ both have bounded geometry than the Sobolev norm \eqref{eq:sobolev_norm} is equivalent to the local one given by
\begin{equation}\label{eq:sobolev_norm_local}
\|s\|_{H^k}^2 \stackrel{\text{equiv}}= \sum_{i=1}^\infty \|\varphi_i s\|^2_{H^k(B_{2\varepsilon}(x_i))},
\end{equation}
where the balls $B_{2\varepsilon}(x_i)$ are domains of normal coordinate charts, the subordinate partition of unity $\{\varphi_i\}$ is such that the derivatives are uniformly bounded (i.e., independent of $i$), we have chosen synchronous framings and $\|\largecdot\|_{H^k(B_{2\varepsilon}(x_i))}$ denotes the usual Sobolev norm on $B_{2\varepsilon}(x_i) \subset \IR^m$. This equivalence enables us to define the Sobolev norms for all $k \in \IR$, see Triebel \cite{triebel_2} and Gro{\ss}e--Schneider \cite{grosse_sobolev}.

\begin{thm}[{\cite[Theorem 2.21]{aubin_nonlinear_problems}}]\label{thm:sobolev_embedding}
Let $E$ be a vector bundle of bounded geometry over a manifold $M^m$ of bounded geometry and without boundary.

Then we have for all values $k-r > m/2$ continuous embeddings
\[H^k(E) \subset C^r_b(E).\]
\end{thm}

\subsection{Finite propagation smoothing operators}

We define the space
\[H^\infty(E) := \bigcap_{k \in \IN_0} H^k(E)\]
and equip it with the obvious \Frechet topology. The Sobolev Embedding Theorem tells us now that we have a continuous embedding
\[H^\infty(E) \hookrightarrow C^\infty_b(E).\]

\begin{lem}
The topological dual of $H^\infty(E)$ is given by
\[H^{-\infty}(E) := \bigcup_{k \in \IN_0} H^{-k}(E).\]
\end{lem}

Equip the space $H^{-\infty}(E)$ with the inductive limit topology $\iota(H^{-\infty}(E), H^\infty(E))$:
\[H^{-\infty}_\iota(E) := \operatorname{\underrightarrow{\lim}} H^{-k}(E).\]
It enjoys the following universal property: a linear map $A \colon H^{-\infty}_\iota(E) \to F$ to a locally convex topological vector space $F$ is continuous if and only if $A|_{H^{-k}(E)}\colon H^{-k}(E) \to F$ is continuous for all $k \in \IN_0$. Note that this topology differs from the weak topology that we might also have put on $H^{-\infty}(E)$.

\begin{defn}[Smoothing operators]
Let $M$ be a manifold of bounded geometry and $E$ and $F$ two vector bundles of bounded geometry over $M$. We will call a continuous linear operator $A \colon H^{-\infty}_\iota(E) \to H^\infty(F)$ a \emph{smoothing operator}.

Note that this is equivalent to $A \colon H^{-k}(E) \to H^l(F)$ being bounded for all $k, l \in \IN_0$
\qed
\end{defn}

Denote by $\IB(H^{-\infty}_\iota(E), H^\infty(E))$ the algebra of all smoothing operators on $E$ and equip it with the countable family of norms $(\| \largecdot \|_{-k,l})_{k,l \in \IN_0}$, where $\|A\|_{-k,l}$ denotes the operator norm of $A \colon H^{-k}(E) \to H^l(F)$. So $\IB(H^{-\infty}_\iota(E), H^\infty(E))$ becomes a \Frechet space\footnote{That is to say, a topological vector space whose topology is Hausdorff and induced by a countable family of semi-norms such that it is complete with respect to this family of semi-norms.}.

The Schwartz Kernel Theorem for regularizing operators together with the bounded geometry of $M$ and the bundles $E$ and $F$ gives the following proposition:

\begin{prop}\label{prop:smoothing_op_kernel}
Let $A\colon H^{-\infty}_\iota(E) \to H^\infty(F)$ be a smoothing operator. Then $A$ is an integral operator with a kernel $k_A \in C_b^\infty(F \boxtimes E^\ast)$. Furthermore, the map
\begin{equation}
\label{eqklf34}
\IB(H^{-\infty}_\iota(E), H^\infty(F)) \to C_b^\infty(F \boxtimes E^\ast)
\end{equation}
associating a smoothing operator its kernel is continuous.
\end{prop}

\begin{defn}[Finite propagation]\label{defn:finite_prop_speed}
Let $A\colon L^2(E) \to L^2(F)$ be an operator. We will say that $A$ has \emph{finite propagation}, if there exists an $R \in \IR_{\ge 0}$ such that we have $\supp As \subset \overline{B_R(\supp s)}$ for all sections $s$ into $E$.

In this case we call the smallest possible value of $R$ the \emph{propagation} of $A$.
\qed
\end{defn}

If $A$ is a smoothing operator, then it has by Proposition \ref{prop:smoothing_op_kernel} a smooth integral kernel $k_A \in C_b^\infty(F \boxtimes E^\ast)$. It is clear that $A$ has propagation at most $R$ if and only if the kernel satisfies $k_A(x,y) = 0$ for all $x,y \in M$ with $d(x,y) > R$. On the other hand, if we have some section $k \in C_b^\infty(F \boxtimes E^\ast)$ with $k(x,y) = 0$ for all $x,y \in M$ with $d(x,y) > R$ for some $R > 0$, then the integral operator $A_k$ defined by it is a smoothing operator with propagation at most $R$. Note that in this case the adjoint operator $A_k^\ast = A_{k^\ast}$ will also be a smoothing operator with propagation at most $R$.

\begin{defn}[Algebraic smooth uniform Roe algebra]\label{defn:alg_smooth_uniform_roe_algebra}
Let $M$ be a manifold of bounded geometry and $E \to M$ a vector bundle of bounded geometry over $M$.

We denote the $^\ast$-algebra of all finite propagation smoothing operators on $E$ by $\IC_{-\infty}^\ast(E)$ and call it the \emph{algebraic smooth uniform Roe algebra of $E$}.
\qed
\end{defn}

The reason why we call $\IC_{-\infty}^\ast(E)$ the algebraic smooth uniform Roe algebra is because it is dense in the algebraic uniform Roe algebra $\IC_u^\ast(E)$.\footnote{The algebraic uniform Roe algebra $\IC_u^\ast(E)$ is defined as the $^\ast$-algebra of all finite propagation, uniformly locally compact operators in $\IB(L^2(E))$. The uniform Roe algebra $C_u^\ast(E)$ is defined as its completion.} This was proved by the author in his Ph.D.~thesis \cite[Lemma 2.57]{engel_phd}.

\subsection{Polynomially decaying quasilocal smoothing operators}
\label{sec:dense_smooth_subalgebras}

The main goal of this section is to define the smooth subalgebra $\Cpol(E)$ of the uniform Roe algebra $C_u^\ast(E)$.

\begin{defn}[Smooth uniform Roe algebra]\label{defnj23wef}
The closure of the algebraic smooth uniform Roe algebra $\uRoe(E)$ under the family of semi-norms $(\|\largecdot\|_{-k,l}, \|\largecdot^\ast\|_{-k,l})_{k,l \in\IN_0}$ is the \emph{smooth uniform Roe algebra} $\C(E)$.
\qed
\end{defn}

\begin{prop}[cf.~{\cite[Section 2.3]{engel_phd}}]\label{prop:smooth_Roe_holomorphically_closed}
The smooth uniform Roe algebra $\C(E)$ is a \Frechet $^\ast$-algebra\footnote{This is an algebra with a topology turning it into a \Frechet space with jointly continuous multiplication and such that the $^\ast$-operation is continuous. Note that we do not require that the semi-norms are sub-multiplicative.}

Furthermore, it is densely included in the uniform Roe algebra $C_u^\ast(E)$ and it is smooth, i.e., it and all matrix algebras over it are closed under holomorphic functional calculus.
\end{prop}

\begin{proof}
The only non-trivial part in showing that $\C(E)$ is a \Frechet $^\ast$-algebra is the joint continuity of the multiplication. So if given smoothing operators $A$ and $B$, we have for all $k,l \in \IN_0$ the diagram
\[\xymatrix{H^{-k}(E) \ar[r]^{AB} \ar[d]_B & H^l(E) \\ H^1(E) \ar@{^{(}->}[r] & H^{-1}(E) \ar[u]^A}\]
from which we get $\|AB\|_{-k,l} \le \|B\|_{-k,1} \cdot C \cdot \|A\|_{-1,l}$, where $C$ is the norm of the inclusion $H^1(E) \hookrightarrow H^{-1}(E)$. From this it follows that multiplication is jointly continuous.

Since the family of semi-norms on $\C(E)$ contains the usual operator norm $\|\largecdot \|_{0,0}$, we have a continuous inclusion with dense image $\C(E) \to C_u^\ast(E)$, because the algebraic smooth uniform Roe algebra $\IC_{-\infty}^\ast(E)$ is dense in $C_u^\ast(E)$; cf. \cite[Lemma 2.57]{engel_phd}.

It remains to show that $\C(E)$ is closed under holomorphic functional calculus, since by \cite[Corollary 2.3]{schweitzer} it will follow from this that all matrix algebras over $\C(E)$ are also closed under it. To do this we will use \cite[Theorem 2.1]{schmitt} which we will recall for the convenience of the reader after this proof. Let $A \in \C(E)$ and let $f(x) = \sum_{i \ge 1} a_i x^i$ be a power series around $0 \in \IC$ with radius of convergence bigger than $\|A\|_{op} = \|A\|_{0,0}$. We have to show that $f(A) \in \C(E)$. For $k,l \in \IN_0$ we have the estimate
\[\|A^{n+2}\|_{-k,l} \le \|A\|_{0,l} \cdot \|A\|_{0,0}^n \cdot \|A\|_{-k,0}\]
and so
\begin{align*}
\|f(A)\|_{-k,l} & \le \sum_{i \ge 1} |a_i| \cdot \|A^i\|_{-k,l}\\
& \le |a_1| \|A\|_{-k,l} + \sum_{i \ge 2} |a_i| \cdot \|A\|_{0,l} \cdot \|A\|_{0,0}^{i-2} \cdot \|A\|_{-k,0}\\
& < \infty.
\end{align*}
Since this holds for every $k,l \in \IN_0$, we conclude that $f(A)$ is a smoothing operator. Applying the same argument to the adjoint $A^\ast$ of $A$, we conclude $f(A) \in \C(E)$.
\end{proof}

\begin{lem}[cf.~{\cite[Theorem 2.1]{schmitt}}]
Let $A$ be a normed algebra with the property: for all $a \in A$ and every power series $f$ around $0 \in \IC$ with radius of convergence bigger than $\|a\|$ (and with $f(0) = 0$ if $A$ is non-unital), we have $f(a) \in A$.

Then $A$ is closed under holomorphic functional calculus.
\end{lem}

Let $L \subset M$ be any subset. We will denote by $\|\largecdot\|_{H^r, L}$ the seminorm on the Sobolev space $H^r(E)$ given by
\[\|u\|_{H^r, L} := \inf \{ \|u^\prime\|_{H^r} \colon u^\prime \in H^r(E), u^\prime = u \text{ on a neighbourhood of }L\}.\]

\begin{defn}[Quasilocal operators, {\cite[Section 5]{roe_index_1}}]\label{defn:quasiloc_ops}
We will call a continuous operator $A \colon H^r(E) \to H^s(F)$ \emph{quasilocal}, if there is a function $\mu\colon \IR_{> 0} \to \IR_{\ge 0}$ with $\mu(R) \to 0$ for $R \to \infty$ and such that for all $L \subset M$ and $u \in H^r(E)$ with $\supp u \subset L$
\begin{equation}
\label{eqdom94}
\|A u\|_{H^s, M - B_R(L)} \le \mu(R) \cdot \|u\|_{H^r}.
\end{equation}
Such a function $\mu$ will be called a \emph{dominating function for $A$}.

We say that $A\colon C_c^\infty(E) \to C^\infty(F)$ is a \emph{quasilocal operator of order $k$}\footnote{Roe calls such operators ``\emph{uniform} operators of order $k$'' in \cite[Definition 5.3]{roe_index_1}.} for some $k \in \IZ$, if $A$ has a continuous extension to a quasilocal operator $H^s(E) \to H^{s-k}(F)$ for all $s \in \IZ$.

A smoothing operator $A\colon H^{-\infty}_\iota(E) \to H^\infty(F)$ will be called \emph{quasilocal}, if $A$ is quasilocal as an operator $H^{-k}(E) \to H^l(F)$ for all $k,l \in \IN_0$.
\qed
\end{defn}

Note that if $A$ is a finite propagation operator, then $A$ is quasilocal.

\begin{lem}[{\cite[Lemma 2.26]{engel_phd}}]\label{lem:limit_quasilocal}
If $(A_i)_i$ is a sequence of quasilocal operators $H^{-k}(E) \to H^l(E)$ converging in the norm $\|\largecdot \|_{-k,l}$ to an operator $A$, then $A$ will also be a quasilocal operator.
\end{lem}

Roe gave in \cite[Proposition 5.2]{roe_index_1} the following estimate: if $\mu_A$ is a dominating function for the operator $A \in \IB(L^2(E))$ and $\mu_B$ one for $B \in \IB(L^2(E))$, then a dominating function for $AB \in \IB(L^2(E))$ is given by
\begin{equation}
\label{eq:mu_AB}
\mu_{AB}(R) = \|A\|_{\mathrm{op}} \cdot 2\mu_B(R/2) + \mu_A(R/2)\big( \|B\|_{\mathrm{op}} + 2 \mu_B(R/2) \big),
\end{equation}
where $\|\largecdot\|_{\mathrm{op}}$ is the operator norm on $\IB(L^2(E))$ and by $\mu_{\largecdot}(R)$ we denote dominating functions for operators $L^2(E) \to L^2(E)$. Since we can always assume $\mu_A(R) \le \|A\|_{\mathrm{op}}$, we can estimate the dominating function for $A^2$ (i.e., in the case $A=B$) from above by
\begin{equation}
\label{eq:mu_A_squared}
\mu_{A^2}(R) \le 5 \|A\|_{\mathrm{op}} \mu_A(R/2).
\end{equation}

\begin{lem}\label{lem:dominating_function_powers}
For every $n \in \IN$ we have for all $R > 0$
\[\mu_{A^{n+1}}(R) \le \sum_{k=1}^n 5^k \|A\|_{\mathrm{op}}^n \mu_A(R/2^k).\]
\end{lem}

\begin{proof}
The case $n = 1$ is Equation \eqref{eq:mu_A_squared} above. By applying Equation \eqref{eq:mu_AB} to $\mu_{A^n A}(R)$ we get inductively the claimed estimate for all $n \in \IN$.
\end{proof}

\begin{defn}\label{defn:Cn}
We define for every $n \in \IN$ a norm on $\uRoe(E)$ by
\[\|A\|_{\mu,n} := \inf \{D > 0\colon \mu_A(R) \le D/R^n \ \forall R > 1\},\]
where $\mu_A(R)$ denotes now the smallest possible choice of dominating function for the operator $A \colon L^2(E) \to L^2(E)$, i.e.,
\[\mu_A(R) = \inf \{C > 0 \colon \|Au\|_{L^2, M-B_R(\supp u)} \le C \|u\|_{L^2} \text{ for all } u \in L^2(E)\}.\]
(Note that we have $\mu_A(R) \le \|A\|_{\op} \le \|A\|_{\op} / R^n$ for all $R \le 1$.)

We let $\Cpoln(E)$ be the closure of the algebraic smooth uniform Roe algebra $\IC_{-\infty}^\ast(E)$ under the family of norms $(\|\largecdot\|_{-k,l}, \|\largecdot^\ast\|_{-k,l}, \|\largecdot\|_{\mu,n}, \|\largecdot^\ast\|_{\mu,n})_{k,l \in\IN_0}$.
\qed
\end{defn}

\begin{lem}\label{lem:Cpoln_Frechet_smooth}
$\Cpoln(E)$ is a \Frechet $^\ast$-algebra which is dense in $\C(E)$ and smooth.
\end{lem}

\begin{proof}
The only non-trivial part of showing that it is a \Frechet $^\ast$-algebra is to show that multiplication is jointly continuous. But using Equation \eqref{eq:mu_AB} we get
\[\|AB\|_{\mu,n} \le \|A\|_{\mathrm{op}}\cdot 2\cdot 2^n \|B\|_{\mu,n} + 2^n \|A\|_{\mu,n} \big( \|B\|_{\mathrm{op}} + 2 \cdot 2^n \|B\|_{\mu,n} \big).\]
Together with the fact that we already know that multiplication is jointly continuous on $\C(E)$, the joint continuity of multiplication on $\Cpoln(E)$ follows.

That $\Cpoln(E)$ is dense in $\C(E)$ follows from the fact that $\IC_{-\infty}^\ast(E) \subset \Cpoln(E)$ is dense in $\C(E)$.

It remains to show that $\Cpoln(E)$ is closed under holomorphic functional calculus; by \cite[Corollary 2.3]{schweitzer} it will follow that the same holds for all matrix algebras over it. By \cite[Lemma 1.2]{schweitzer} it suffices to show that $\Cpoln(E)$ is inverse closed, and for this it suffices by \cite[Lemma 3.38]{GBVF} to show that $(\id - B)^{-1} \in \Cpoln(E)$ if $B \in \Cpoln(E)^+$ with $\|B\|_{\mathrm{op}} < 1/(2^{n+1}\cdot 5)$. Here $\largecdot^+$ denotes unitization. We know from Proposition \ref{prop:smooth_Roe_holomorphically_closed} that $\C(E)$ is inverse closed, so it suffices to give an estimate for $\|(\id - B)^{-1}\|_{\mu,n}$.

So let $B \in \Cpoln(E)^+$ with $\|B\|_{\mathrm{op}} < 1/(2^{n+1} \cdot 5)$ be given. Then $(\id - B)^{-1} = \sum_{k \ge 0} B^k$ with convergence in operator norm and we have $\|(\id - B)^{-1} - \sum_{k = 0}^{N} B^k\|_{\mathrm{op}} < (2^{n+1} \cdot 5)^{-N}.$

To get an estimate on the dominating function of $(\id - B)^{-1}$ we will use the argument from the proof of Lemma \ref{lem:limit_quasilocal}: The sequence $\left(\sum_{k = 0}^{N} B^k\right)_{\! N \in \IN_0}$ approximates $(\id - B)^{-1}$ and with every step (i.e., increasing $N$ by one) the error of the approximation is multiplied by $1/(2^{n+1} \cdot 5)$. So if we denote by $\mu_{N}(R)$ the dominating function of $\sum_{k = 0}^{N} B^k$, we get
\[\mu_{(\id - B)^{-1}}(R) \le (2^{n+1} \cdot 5)^{-N} + \mu_N(R)\]
for all $R \in [2^{N-1}, 2^{N}]$ and all $N \ge 1$. Note that for $R \le 1$ we can always use the estimate $\mu_A(R) \le \|A\|_{\op} / R^n$, i.e., we have to find an estimate for $\mu_{(\id - B)^{-1}}(R)$ only for $R \ge 1$.

We will find a function $\nu_N(R)$ with $\mu_N(R) \le \nu_N(R)$ and $\nu_{N+1}(2^{N+1}) \le \nu_N(2^N) \cdot 1/2^n$, because then we will get $\mu_{(\id - B)^{-1}}(2^N) \le (2^{n+1} \cdot 5)^{-N} + \nu_1(1)\cdot (2^n)^{-N}.$ From this we get $\mu_{(\id - B)^{-1}}(R) \le R^{-(n+1)} \cdot 5^{-\log_2 R} + \nu_1(1) \cdot R^{-n}$ which gives the final estimate
\[\|(\id - B)^{-1}\|_{\mu,n} \le \max \{ \|(\id - B)^{-1}\|_{\op}, 1 + \nu_1(1)\}.\]

We have $\mu_{N+1}(R) \le \sum_{k=0}^{N+1} \mu_{B^k}(R)$. Applying Lemma \ref{lem:dominating_function_powers} we get (note $\mu_{B^0}(R) \equiv 0$)
\begin{align*}
\mu_{N+1}(R) & \le \mu_B(R) + \sum_{k=1}^N \sum_{l=1}^{k} 5^l \|B\|_{\mathrm{op}}^{k} \mu_B(R/2^l)\\
& = \mu_B(R) + \sum_{l=1}^{N} 5^l \mu_B(R/2^l) \cdot \underbrace{\sum_{k=l}^N \|B\|_{\mathrm{op}}^{k}}_{= \frac{\|B\|_{\mathrm{op}}^l - \|B\|_{\mathrm{op}}^{N+1}}{1-\|B\|_{\mathrm{op}}}}\\
& \le \mu_B(R) + \sum_{l=1}^{N} 5^l \mu_B(R/2^l) \|B\|^l_{\mathrm{op}} / (1-\|B\|_{\mathrm{op}}).\\
& < \mu_B(R) + \sum_{l=1}^{N} 2^{-(n+1)l} \mu_B(R/2^l) / (1-\|B\|_{\mathrm{op}}),
\end{align*}
where the last estimate follows from $\|B\|^l_{\mathrm{op}} < 1/(2^{n+1}\cdot 5)$. From $\mu_B(R) \le \|B\|_{\mu,n} \cdot R^{-n}$ for all $R \ge 1$ we get for $R > 2^N$
\begin{align*}
\mu_{N+1}(R) & \le \|B\|_{\mu,n} \cdot R^{-n} + \sum_{l=1}^N \underbrace{2^{-(n+1)l} (R/2^l)^{-n}}_{= 2^{-l} \cdot R^{-n}} \frac{\|B\|_{\mu,n}}{1-\|B\|_{\mathrm{op}}}\\
& \le \left( \|B\|_{\mu,n} + \frac{\|B\|_{\mu,n}}{1-\|B\|_{\mathrm{op}}} \right) \cdot R^{-n}\\
& =: \nu_{N+1}(R)
\end{align*}
and the functions $\nu_{N+1}(R)$ have the needed property $\nu_{N+1}(2^{N+1}) \le \nu_N(2^N) \cdot 1/2^n$. Note that these functions do not even depend on $N$. So finally we get the sought estimate
\begin{equation}
\label{eqj4re}
\|(\id - B)^{-1}\|_{\mu,n} \le \max \Big\{ \|(\id - B)^{-1}\|_{\op}, 1 + \|B\|_{\mu,n} + \frac{\|B\|_{\mu,n}}{1-\|B\|_{\mathrm{op}}} \Big\}.
\end{equation}

Note that we actually not only have to show that $\|(\id - B)^{-1}\|_{\mu,n}$ is finite, but that this operator lies in $\Cpoln(E)$. This means that we also have to show that $(\id - B)^{-1}$ is approximable by operators of finite propagation in the norm $\|\largecdot\|_{\mu,n}$. But this follows from \eqref{eqj4re} since this estimate shows that inversion is continuous and therefore, since $B$ is approximable by finite propagation operators, the same will be also true for $(\id - B)^{-1}$.
\end{proof}

\begin{defn}\label{defn:Cpol}
\[\Cpol(E) := \bigcap_{n \in \IN} \Cpoln(E)\]
equipped with the obvious \Frechet topology.
\qed
\end{defn}

Equivalently we may describe $\Cpol(E)$ as being the closure of $\IC_{-\infty}^\ast(E)$ under the family of norms $(\|\largecdot\|_{-k,l}, \|\largecdot^\ast\|_{-k,l}, \|\largecdot\|_{\mu,n}, \|\largecdot^\ast\|_{\mu,n})_{k,l \in\IN_0, n \in \IN}$.

Elements of $\Cpol(E)$ are smoothing operators such that they and their adjoint admit dominating functions which decay faster than any polynomial. It is an open question if we can \emph{define} $\Cpol(E)$ by this condition. The problem lies in showing that quasilocal operators are approximable by finite propagation operators; cf.~\cite[Section~6.2]{engel_indices_UPDO}. In the case of $M$ having polynomial volume growth we will give a solution to this problem in Corollary~\ref{corjknd9023}.

\begin{prop}\label{prop:Cpol_Frechet_smooth}
$\Cpol(E)$ is a \Frechet $^\ast$-algebra, dense in every $\Cpoln(E)$ and smooth.
\end{prop}

\begin{proof}
Use Lemma \ref{lem:Cpoln_Frechet_smooth} and that $\IC_{-\infty}^\ast(E) \subset \Cpol(E)$ is dense in $\Cpoln(E)$.
\end{proof}

\begin{cor}\label{cor:smooth_subalgebras}
Every map in the chain of dense, continuous inclusions
\[\Cpol(E) \to \cdots \to C^\ast_{\mathrm{pol},n+1}(E) \to \Cpoln(E) \to \cdots \to \C(E) \to C_u^\ast(E)\]
induces isomorphisms on $K$-theory.
\end{cor}

\begin{rem}\label{rem:defn_seminorms_uRoe}
In the above discussion we have only treated the dominating function for operators $L^2(E) \to L^2(E)$, but for Lemma \ref{lem:decay_integral_kernel} in the next subsection we need to treat also the case $H^{-k}(E) \to H^l(E)$. So let us define on $\IC_{-\infty}^\ast(E)$ the norms
\begin{equation}
\label{eqkjnd90324}
\|A\|_{\mu,n}^{-k,l} := \inf \{D > 0\colon \mu_A^{-k,l}(R) \le D/R^n \ \forall R > 1\},
\end{equation}
where $\mu_A^{-k,l}(R)$ denotes the smallest possible choice of dominating function for the operator $A \colon H^{-k}(E) \to H^l(E)$. We redefine now $\Cpoln(E)$ to be the closure of $\IC_{-\infty}^\ast(E)$ under the family of norms $(\|\largecdot\|_{-k,l}, \|\largecdot^\ast\|_{-k,l}, \|\largecdot\|_{\mu,n}^{-k,l}, \|\largecdot^\ast\|_{\mu,n}^{-k,l})_{k,l \in\IN_0}$ and set $\Cpol(E) := \bigcap_{n \in \IN} \Cpoln(E)$. We have to argue now why Lemma \ref{lem:Cpoln_Frechet_smooth} still holds using this new definition.

The crucial point in the proof of Lemma \ref{lem:Cpoln_Frechet_smooth} is the inverse closedness, i.e., to give an estimate for $\|(\id - B)^{-1}\|_{\mu,n}^{-k,l}$. For this we will use the same argument as in the proof of Proposition \ref{prop:smooth_Roe_holomorphically_closed}: we factor the operator $B^{r+2}\colon H^{-k}(E) \to H^l(E)$ as
\begin{equation}
\label{eq:factoring_powers_B}
H^{-k}(E) \stackrel{B}\longrightarrow L^2(E) \stackrel{B^r}\longrightarrow L^2(E) \stackrel{B}\longrightarrow H^l(E).
\end{equation}
We get that $(\id - B)^{-1} = \sum_{r\ge 0} B^r$ holds also with convergence in every $\|\largecdot\|_{-k,l}$-norm and the error of the approximation of $(\id - B)^{-1}$ by this sum is multiplied by $1/(2^{n+1} \cdot 5)$ on every step (i.e., increasing $r$ by one). We use now the same notation as in the proof of Lemma \ref{lem:Cpoln_Frechet_smooth} and recapitulate the argument for inverse closedness from there. We have to give an estimate for the dominating functions $\mu_N^{-k,l}(R)$ of $\sum_{r=0}^N B^r$. At this point in the proof of Lemma \ref{lem:Cpoln_Frechet_smooth} we applied Lemma \ref{lem:dominating_function_powers} to get a suitable estimate. We do the same here, but incorporate the Factoring \eqref{eq:factoring_powers_B} of $B^{r+2}$: by \cite[Proposition 5.2]{roe_index_1} we get
\[\mu_{B^{r+1}}^{-k,l}(R) \le \|B\|^{r-1}_{0,0} \|B\|_{0,l} \cdot 2\mu_B^{-k,0}(R/2) + \mu_{B^r}^{0,l}(R/2) \cdot 3\|B\|_{-k,0}\]
and into this estimate we plug in the estimate
\[\mu_{B^r}^{0,l}(R/2) \le \|B\|_{0,l} \cdot 2\mu^{0,0}_{B^{r-1}}(R/4) + \mu_B^{0,l}(R/4) \cdot 3\|B\|_{0,0}^{r-1}.\]
For $\mu^{0,0}_{B^{r-1}}(R/4)$ the same estimates that we used in the proof of Lemma \ref{lem:Cpoln_Frechet_smooth}.
\qed
\end{rem}

\subsection{General kernel estimates}
\label{subsec0987}

In this section the manifold $M$ and the vector bundle $E \to M$ are always assumed to have bounded geometry.

\begin{lem}[{cf.~\cite[Proposition 5.4]{roe_index_1}}]\label{lem:decay_integral_kernel}
Let $A\colon H^{-\infty}_\iota(E) \to H^\infty(E)$ be a quasi-local smoothing operator. Then we have for its integral kernel $k_A(x,y) \in C_b^\infty(E \boxtimes E^\ast)$ the following estimates:
\begin{align}
\int_{M - B_R(y)} \|\nabla^r_x \nabla^l_y k_A(x,y)\|^2 dx & < C \cdot \mu_A^{-k,r}(R)^2 \text{ and }\label{eqkj200}\\
\int_{M-B_R(x)} \|\nabla^l_x \nabla^r_y k_A(x,y)\|^2 dy & < D \cdot \mu_{A^\ast}^{-k,r}(R)^2.\notag
\end{align}
The estimates hold for every $r,l \in \IN_0$ and every $k > l + \dim(M)/2$, and the constants $C$ and $D$ do not depend on the operator $A$, scale $R$ or points $x$, $y$ (but they do depend on the choice of $k$, $r$, $l$).
\end{lem}

\begin{proof}
Let $v \in E_y$ for $y \in M$ be a vector of unit norm and let $\delta_v$ be the functional on $C_b^0(E)$ given by $s \mapsto \langle s, v\rangle_y$. Due to the Sobolev embedding theorem (cf.~Theorem~\ref{thm:sobolev_embedding}), the functional $\delta_v$ belongs to $H^{-k}(E)$ for every $k > \dim(M)/2$. Since $A$ maps $H^{-k}(E)$ quasi-locally to $H^0(E)$, we have $\|A \delta_v\|_{H^0,M-B_r(y)} \le \mu_A^{-k,0}(R) \cdot \|\delta_v\|_{H^{-k}}$. By definition we have the chain of equalities
\[\|A\delta_v\|_{H^0,M-B_r(y)}^2 = \int_{M-B_R(y)} \|(A\delta_v)(x)\|^2 \, dx = \int_{M-B_R(y)} \|\langle k_A(x,y),v\rangle_y\|^2 \, dx.\]
Going now through an orthonormal basis of $E_y$ and summing up, we conclude
\[\int_{M-B_R(y)} \| k_A(x,y) \|^2 \, dx \le \mu_A^{-k,0}(R)^2 \cdot \max_{v \in E_y, \|v\|=1} \|\delta_v\|^2_{H^{-k}} \cdot \dim(E).\]

The get the corresponding estimates with the derivatives, we choose at the point $y \in M$ unit vector fields $X_1, \ldots, X_l$ on $M$ and a unit vector $v \in E_y$, and look at the functional $\delta_v^{X_1, \ldots, X_l}$ on $C_b^l(E)$ given by $s \mapsto \langle \nabla_{X_l} \cdots \nabla_{X_1} s, v\rangle_y$. It belongs to $H^{-k}(E)$ for every $k > l+\dim(M)/2$ and we can carry out the above arguments for the case of no derivatives, where we exploit that $A$ maps $H^{-k}(E)$ quasi-locally to $H^r(E)$.

To get the estimates with the roles of $x$ and $y$ interchanged, we just have to pass to the adjoint operator since $k_A(x,y)^\ast = k_{A^\ast}(y,x)$. Note that Roe gives a separate argument for this case, relying on the fact that $A$ maps quasi-locally to $H^\infty(E)$.
\end{proof}

Chen--Wei defined a metric version of Property (RD) by demanding that a certain \Frechet space, which they call $BS_2(X)$, is contained in the uniform Roe algebra \cite{chen_wei}. The following proposition can be regarded as proving that $\Cpol(M)$ which we defined in the previous section (Definition~\ref{defn:Cpol} and Remark~\ref{rem:defn_seminorms_uRoe}) satisfies in a smooth way (i.e., including derivatives, and using integrals instead of summation) the estimates in order to be a subspace of $BS_2(X)$.\footnote{Yu \cite[Page~470]{yu_cyclic} also considered a similar space as $BS_2(X)$, but he used sub-exponential instead of polynomial decay.} But note that $\Cpol(M)$ behaves very differently from $BS_2(X)$, e.g., $BS_2(X)$ is an algebra (i.e., closed under composition) only if $X$ has polynomial growth \cite[Theorem~3.1]{chen_wei}, but $\Cpol(M)$ is always an algebra.

\begin{prop}
\label{lem_decay_kernel_weights}
Let $A\colon H^{-\infty}_\iota(E) \to H^\infty(E)$ be a quasi-local smoothing operator.

Then we have for its integral kernel $k_A(x,y) \in C_b^\infty(E \boxtimes E^\ast)$ the following estimates:
\begin{align}
\int_M \|\nabla^r_x \nabla^l_y k_A(x,y)\|^2 \cdot d(x,y)^N dx & < C^\prime \cdot \big(\|A\|^{-k,r}_{\mu,N/2+1}\big)^2 \text{ and }\label{eqk876j200}\\
\int_M \|\nabla^l_x \nabla^r_y k_A(x,y)\|^2 \cdot d(x,y)^N dy & < D^\prime \cdot \big(\|A^\ast\|^{-k,r}_{\mu,N/2+1}\big)^2.\notag
\end{align}
The estimates hold for every $r,l \in \IN_0$, every $k > l + \dim(M)/2$ and every $N \in \IN_0$, and the constants $C^\prime$ and $D^\prime$ do not depend on the operator $A$ or points $x$, $y$ (but they do depend on the choice of $k$, $r$, $l$, and $N$).
\end{prop}

\begin{proof}
Let us first treat the case $r,l = 0$. Let $N \in \IN_0$ be given. In the following estimate we have used \eqref{eqkj200}, and $k > \dim(M)/2$ is fixed:
\begin{align*}
\int_2^\infty \int_{M - B_R(y)} \|k_A(x,y)\|^2 dx \cdot R^N dR & < \int_2^\infty C \cdot \mu_A^{-k, 0}(R)^2 \cdot R^N dR\\
& \le \int_2^\infty C \cdot \big(\|A\|^{-k,0}_{\mu,N/2+1}\big)^2 / R^{N+2} \cdot R^N dR\\
& = C/2 \cdot \big(\|A\|^{-k,0}_{\mu,N/2+1}\big)^2.
\end{align*}
We can now rearrange the left hand side of the above estimate to get the left hand side of the following estimate, and then do an elementary estimate to arrive at
\[\int_M \int_2^{d(x,y)} \|k_A(x,y)\|^2 \cdot R^N dR\, dx \ge \int_M \|k_A(x,y)\|^2 \cdot \frac{d(x,y)^N}{N+1} dx\]
and so we get the first claimed estimate
\[\int_M \|k_A(x,y)\|^2 \cdot d(x,y)^N dx \le C^\prime \cdot \big(\|A\|^{-k,0}_{\mu,N/2+1}\big)^2.\]
The estimate for the role of $x$ and $y$ interchanged follows from applying the above estimate to the adjoint operator, and the estimates including the higher derivatives of the integral kernel follow analogously from \eqref{eqkj200}.
\end{proof}

Let us strengthen the estimates from the previous proposition by estimating additionally the rate of decay.

\begin{cor}
Let $A\colon H^{-\infty}_\iota(E) \to H^\infty(E)$ be a quasi-local smoothing operator.

Then we have for its integral kernel $k_A(x,y) \in C_b^\infty(E \boxtimes E^\ast)$ the following estimates:
\begin{align}
\int_{M-B_R(y)} \|\nabla^r_x \nabla^l_y k_A(x,y)\|^2 \cdot d(x,y)^N dx & \le C_0 \cdot \mu_{A^\ast_{r,l}}^{0,s}(R) \cdot \|A\|^{-k,r}_{\mu,N+1} \text{ and }\label{eqoiuer6j200}\\
\int_{M-B_R(x)} \|\nabla^l_x \nabla^r_y k_A(x,y)\|^2 \cdot d(x,y)^N dy & \le D_0 \cdot \mu_{A_{l,r}}^{0,s}(R) \cdot \|A^\ast\|^{-k,r}_{\mu,N+1}.\notag
\end{align}
The estimates hold for every $r,l \in \IN_0$, every $k > l + \dim(M)/2$, $s > \dim(M)/2$ and every $N \in \IN_0$, and the constants $C_0$ and $D_0$ do not depend on the operator $A$ or points $x$, $y$ (but they do depend on the choice of $k$, $r$, $l$, $s$ and $N$).

The operator $A_{l,r}$ is the quasi-local smoothing operator given by the kernel $\nabla^l_x \nabla^r_y k_A(x,y)$.
\end{cor}

\begin{proof}
The idea for the following proof comes from the second part of John Roe's proof of \cite[Proposition 5.4]{roe_index_1}.

The above Proposition~\ref{lem_decay_kernel_weights} states that the function $\phi_{y,N}(\largecdot) := k_{A^\ast}(\largecdot, y) \cdot d(\largecdot, y)^N$ is square-integrable. Since $A$ maps $H^0(E)$ quasi-locally to $H^s(E)$ for all $s > \dim(M) / 2$, and the latter is continuously embedded in $C_b^0(E)$ by the Sobolev Embedding Theorem~\ref{thm:sobolev_embedding}, we conclude $|(Au)(y)| \le E \cdot \mu_A^{0,s}(R) \cdot \|u\|_{H^0}$ for all square-integrable functions $u$ on $M$ supported outside of $B_R(y)$, and $E$ is the constant coming from the Sobolev embedding theorem. Setting
\[\phi^\prime_{y,N}(x) := \begin{cases}k_{A^\ast}(x, y) \cdot d(x, y)^N & \text{for } d(x,y) > R\\
0 & \text{otherwise}\end{cases}\]
we therefore conclude
\[\int_{M-B_R(y)} \|k_A(y,x)\|^2 \cdot d(y,x)^N dx \le E \cdot \mu_A^{0,s}(R) \cdot \sqrt{D^\prime} \cdot \|A^\ast\|^{-k,0}_{\mu,N+1}\]
for every $k > \dim(M) / 2$.

As usual, passing to the adjoint operator proves the estimate with the roles of $x$ and $y$ changed, and the estimates incorporating derivatives are derived analogously (but here one first has to prove that the operator given by the kernel $\nabla^l_x \nabla^r_y k_A(x,y)$ is indeed a quasi-local smoothing operator).
\end{proof}

\subsection{Space of integral kernels on manifolds of polynomial growth}

This section provides the main analytical meat for the proof of Theorem F. Our main task in this section is to relate the norms on $\Cpol(E)$, which are certain operator norms, to norm on the integral kernels of such operators (which are so-called unconditional norms). The main result in this section, Theorem~\ref{thmjnksd9023}, is exactly of this kind.

\begin{defn}\label{defnrtzu}
Let $k \in C_b^\infty(E \boxtimes E^\ast)$ be supported in a uniform neighbourhood of the diagonal. We define for it the semi-norms
\begin{equation}
\label{eqjnk8743}
\|k\|_{1,x}^{r,l,N} := \sup_{y \in M} \int_M \|\nabla_x^r \nabla_y^l k(x,y)\| \cdot d(x,y)^N dx,
\end{equation}
and analogously we define the semi-norms $\|k\|_{1,y}^{r,l,N}$ where we take the supremum over all $x \in M$ and integrate with respect to $y \in M$.

We define $\WDpol(E \boxtimes E^\ast)$ as the completion of the space of all functions $k \in C_b^\infty(E \boxtimes E^\ast)$ which are supported in a uniform neighbourhood of the diagonal under the above defined family of semi-norms (i.e., all $r,l,N \in \IN_0$ and both $x$,$y$ as subscripts).
\qed
\end{defn}

The Sobolev embedding theorem (see, e.g., Aubin \cite[Theorem 2.21]{aubin_nonlinear_problems}) can be used to show that we have a continuous inclusion
\begin{equation}
\label{eqo24ew}
\WDpol(E \boxtimes E^\ast) \subset C_b^\infty(E \boxtimes E^\ast).
\end{equation}

\begin{lem}
$\WDpol(E \boxtimes E^\ast)$ is a Fr\'{e}chet-$^\ast$-algebra with jointly continuous multiplication (which is given by convolution).
\end{lem}

\begin{proof}
The semi-norms $\|\largecdot\|_{1,x}^{0,0,0}$ and $\|\largecdot\|_{1,y}^{0,0,0}$ are norms and therefore $\WDpol(E \boxtimes E^\ast)$ is a Hausdorff space. It remains to treat the multiplication. For the sake of simplicity, let us only do the case $r,l = 0$ and the subscript $x$. Then we have the estimate
\begin{align*}
\|k \ast l\|^{0,0,N}_{1,x} & = \int_M \Big\| \int_M k(x,z) l(z,y) dz \Big\| \cdot d(x,y)^N dx\\
& \le \int_M \int_M \Big\| k(x,z) \Big\| \cdot \Big\| l(z,y) \Big\| \cdot d(x,y)^N dz dx.
\end{align*}
Using the triangle inequality $d(x,y) \le d(x,z) + d(z,y)$ and then treating the summands of $(d(x,z)+d(z,y))^N = \sum_{N_1 + N_2 = N} C(N_1,N_2) d(x,z)^{N_1}d(z,y)^{N_2}$ separately, we have
\begin{align*}
\int_M \int_M \Big\| k(x,z) \Big\| & \cdot \Big\| l(z,y) \Big\| \cdot C(N_1,N_2) d(x,z)^{N_1}d(z,y)^{N_2} dz dx\\
& \le \int_M \|k\|_{1,x}^{0,0,N_1} \cdot \Big\|l(z,y)\Big\| \cdot C(N_1,N_2) d(z,y)^{N_2} dz\\
& \le \|k\|_{1,x}^{0,0,N_1} \cdot \|l\|_{1,x}^{0,0,N_2} \cdot C(N_1,N_2),
\end{align*}
where we have changed the order of integration for the first inequality.

This shows that $\|k \ast l\|^{0,0,N}_{1,x}$ is finite, and furthermore, it shows that multiplication is jointly continuous. This finishes this proof.
\end{proof}

\begin{lem}\label{lem8923c}
We have a continuous inclusion $\WDpol(E \boxtimes E^\ast) \to \C(E)$.\footnote{See Definition~\ref{defnj23wef} for the definition of $\C(E)$.}
\end{lem}

\begin{proof}
Let $A_k$ be the operator corresponding to a kernel $k$. Then the operator norm of $A_k$ can be estimated by
\begin{equation}
\label{eqnj9xcv2}
\|A_k\|_\op \le \sup_{x \in M} \Big( \int_M \|k(x,y)\| dy \Big)^{1/2} \cdot \sup_{y \in M} \Big( \int_M \|k(x,y)\| dx \Big)^{1/2}.
\end{equation}
Correspondingly, the norm of $A_k\colon H^{-r}(E) \to H^l(E)$ can be similarly estimated by using the appropriate derivatives of the kernel. So we see that we can estimate the norms of $A_k \in \C(E)$ by the norms \eqref{eqjnk8743} of $k \in \WDpol(E \boxtimes E^\ast)$.\footnote{In fact, we only use here the norms with $N=0$.}
\end{proof}

\begin{defn}
A Riemannian manifold $M$ is said to have polynomial volume growth, if there exist $V,T \in \IR_{> 0}$ such that $\vol B_R(y) \le V (R+1)^T$ for all $y \in M$.
\qed
\end{defn}

\begin{prop}\label{propsdj2309099}
Let $M$ have bounded geometry and polynomial volume growth.

Let $A$ be a quasi-local smoothing operator with $\|A\|_{\mu,n}^{-k,l} < \infty$ for every $k,l,n \in \IN_0$.\footnote{See~\eqref{eqkjnd90324} for the definition of the norms $\|\largecdot\|_{\mu,n}^{-k,l}$.} Then for every $y \in Y$ we have for its integral kernel the estimates
\[\int_M \|\nabla^r_x \nabla^l_y k(x,y)\| \cdot d(x,y)^N dx < C_A(r,l,N) < \infty\]
for all $r$, $l$ and $N$. The constants $C_A(r,l,N)$ are independent of the point $y$.

If $\|A^\ast\|_{\mu,n}^{-k,l} < \infty$, then we have the corresponding estimates with $x$ and $y$ interchanged.
\end{prop}

\begin{proof}
We rewrite the integral over the manifold $M$ as the sum of the integrals over the annuli $R \le d(x,y) \le R+1$ for $R \in \IN_0$. On the annulus $R \le d(x,y) \le R+1$ we use the Cauchy--Schwarz inequality, the estimate $d(x,y) \le R+1$ and the estimate from above of its volume by $V (R+2)^T$, to get the estimate
\[\Big( \int_R^{R+1} \|\nabla^r_x \nabla^l_y k(x,y)\| \cdot d(x,y)^N dx \Big)^2 \le \int_R^{R+1} \|\nabla^r_x \nabla^l_y k(x,y)\|^2 \cdot (R+1)^{2N} dx \cdot V (R+2)^T.\]
Now we use \eqref{eqkj200} to estimate $\int_R^{R+1} \|\nabla^r_x \nabla^l_y k(x,y)\|^2 dx$ from above by $C \cdot \mu_A^{-k,r}(R)^2$. Since $\|A\|_{\mu,n}^{-k,l} < \infty$ by assumption, we have $\mu_A^{-k,r}(R) \le E/R^n$ for some constant $E > 0$ and a choice of $n$ which we will fix later.\footnote{Actually, by definition of $\|\largecdot\|_{\mu,n}^{-k,l}$ we can choose $E = \|A\|_{\mu,n}^{-k,l}$.} So we get
\[\Big( \int_R^{R+1} \|\nabla^r_x \nabla^l_y k(x,y)\| \cdot d(x,y)^N dx \Big)^2 \le C E^2 / R^{2n} \cdot (R+1)^N \cdot V (R+2)^T\]
and from this our final estimate
\begin{align*}
\int_M \|\nabla^r_x \nabla^l_y k(x,y)\| \cdot d(x,y)^N dx & = \sum_{R=0}^\infty \int_R^{R+1} \|\nabla^r_x \nabla^l_y k(x,y)\| \cdot d(x,y)^N dx\\
& \le \sum_{R=0}^\infty \sqrt{C} E / R^n \cdot (R+1)^{N/2} \cdot \sqrt{V} (R+2)^{T/2}\\
& < \infty,
\end{align*}
where we know now that we have to choose $n > N/2 + T/2 + 1$.

Under the assumption $\|A^\ast\|_{\mu,n}^{-k,l} < \infty$ we get the estimate for $x$ and $y$ interchanged by the same argument, but using from Lemma~\ref{lem:decay_integral_kernel} the second estimate.
\end{proof}

\begin{cor}\label{cor9023932}
Let $M$ have polynomial volume growth. Then we have a continuous inclusion $\Cpol(E) \to \WDpol(E \boxtimes E^\ast)$.
\end{cor}

\begin{proof}
From Proposition~\ref{propsdj2309099} it follows that we have a map $\uRoe(E) \to \WDpol(E \boxtimes E^\ast)$. It remains to look more closely into the proof of Proposition~\ref{propsdj2309099} to see that this map is continuous and therefore extends to the completion $\C(E)$: the constant $C_A(r,l,N)$ derived in the proof is given by
\[\sum_{R=0}^\infty \sqrt{C} E / R^n \cdot (R+1)^{N/2} \cdot \sqrt{V} (R+2)^{T/2}\]
and we can choose $E = \|A\|_{\mu,n}^{-k,l}$, where we fix an $n > N/2 + T/2 + 1$. This observation finishes this proof.
\end{proof}

\begin{thm}\label{thmjnksd9023}
Let $M$ be a manifold of bounded geometry and of polynomial volume growth and let $E \to M$ be a vector bundle over $M$ of bounded geometry.

Then we have an isomorphism
\[\Cpol(E) \xrightarrow{\cong} \WDpol(E \boxtimes E^\ast).\]
\end{thm}

\begin{proof}
We know from Corollary~\ref{cor9023932} that the map $\Cpol(E) \to \WDpol(E \boxtimes E^\ast)$ is continuous. So we have to investigate the inverse map, i.e., the map given by assigning a kernel its integral operator. From Lemma~\ref{lem8923c} we know that this inverse map is continuous if regarded as mapping into $\C(E)$. So it remains to show that this inverse map is also continuous for the norms $\|\largecdot\|_{\mu,n}^{-k,l}$ and $\|\largecdot\|_{\mu,n}^{-k,l}$ on $\Cpol(E)$. We will only do the argument for $\|\largecdot\|_{\mu,n}^{-k,l}$ since the argument for the corresponding norm of the adjoint is completely analoguous. Furthermore, for the sake of simplicity, we will only treat the case $k,l = 0$ since the other cases are similar.

So given $k \in \WDpol(E \boxtimes E^\ast)$ we have to bound the dominating function $\mu_{A_k}^{0,0}(R)$ of the corresponding integral operator. For $R \in \IN$ let $A_R$ be an operator of propagation~$\le R$. From Definition~\ref{eqdom94} of the dominating function, we conclude that we have the estimate $\mu_{A_k}^{0,0}(R) \le \|A_k - A_R\|_\op$ for every $R \in \IN$, i.e., to give a good bound on the dominating function of an operator means to give a good approximation of it by operators of small propagation. To get $A_R$ we will just chop the integral kernel $k$ down to the neighbourhood $d(x,y) \le R$ of $M \times M$. This is not a smooth kernel anymore, but using suitable partitions of unity this can be rectified. So let us therefore just work with this chopped kernel, and let us denote the resulting kernel of $A_k - A_R$ by $k^R$, i.e, $k^R(x,y) = k(x,y)$ for all $d(x,y) > R$ and $k^R(x,y) = 0$ otherwise. Let us do the following estimates:
\begin{align*}
R^N \cdot \int_M \|k^R(x,y)\| dy & \le \int_M \|k^R(x,y)\|\cdot d(x,y)^N dy\\
& \le \int_M \|k(x,y)\|\cdot d(x,y)^N dy\\
& \le \|k\|_{1,y}^{0,0,N}.
\end{align*}
Combining this with \eqref{eqnj9xcv2} we arrive at the estimate
\[\|A_k - A_R\|_\op = \|A_{k^R}\|_\op \le \sqrt{\|k\|_{1,y}^{0,0,N} / R^N} \cdot \sqrt{\|k\|_{1,x}^{0,0,N} / R^N}\]
which finishes the proof since it means $\|A_k\|_{\mu,N}^{0,0} \le \sqrt{\|k\|_{1,y}^{0,0,N} \cdot \|k\|_{1,x}^{0,0,N}}$.
\end{proof}

As a by-product of the above analysis we are now able to solve, at least partially, the question whether quasi-local operators are approximable by operators of finite propagation. The following corollary solves this question in the case $M$ has polynomial volume growth and the dominating function of the quasi-local operator decays sufficiently fast. See also the discussion in \cite[Section~6.2]{engel_indices_UPDO}.

\begin{cor}\label{corjknd9023}
Let $M$ have bounded geometry and polynomial volume growth.

Then every quasi-local smoothing operator $A$ with $\|A\|_{\mu,n}^{-k,l} < \infty$ for every $k,l,n \in \IN_0$ is contained in $\Cpol(E)$, i.e., is approximable by finite propagation operators.
\end{cor}

\begin{proof}
In the above proof of Theorem~\ref{thmjnksd9023} we have derived the estimate
\begin{equation}
\label{eqnjkd02323}
\|A - A_R\|_\op \le \sqrt{\|k\|_{1,y}^{0,0,N} / R^N} \cdot \sqrt{\|k\|_{1,x}^{0,0,N} / R^N}
\end{equation}
for every $N \in \IN$, where $k$ is the integral kernel of $A$. Although in that theorem we assume $A \in \Cpol(E)$, its proof relies on Proposition~\ref{propsdj2309099} which only assumes that $\|A\|_{\mu,n}^{-k,l} < \infty$, i.e., for the above estimate we do not need $A$ to be alrady approximable by finite propagation operators. Now \eqref{eqnjkd02323} shows that the operator $A_R$ approximate $A$ in operator norm and each $A_R$ has propagation $\le R$. Together with the corresponding estimates for the derivatives, we conclude $A \in \C(E)$.

It remains to treat the norms $\|\largecdot\|_{\mu,n}^{-k,l}$ and we will again only discuss the case $k,l = 0$. We will show that $\|A - A_R\|_{\mu,n}^{0,0}$ goes to $0$ as $R \to \infty$. From the discussion in the proof of Theorem~\ref{thmjnksd9023} we get that a dominating function $\mu_{A-A_R}^{0,0}(R^\prime)$ for $A-A_R$ is given by $\|(A-A_R)-A_{R,R^\prime}\|_\op$, where $A_{R,R^\prime}$ is an operator with propagation $\le R^\prime$. Again just chopping down the integral kernel of $A - A_R$, we can use for $A_{R,R^\prime}$ the integral kernel given by $k^{R^\prime}(x,y) = k(x,y)$ for all $R \le d(x,y) \le R^\prime$ and $k^{R^\prime}(x,y) = 0$ otherwise, if $R^\prime \ge R$. If $R^\prime \le R$ we use $A_{R,R^\prime} = 0$. We get
\[\mu_{A-A_R}^{0,0}(R^\prime) \le \begin{cases} \sqrt{\|k\|_{1,y}^{0,0,N} / R^N} \cdot \sqrt{\|k\|_{1,x}^{0,0,N} / R^N} & \text{for } R^\prime \le R,\\ \sqrt{\|k\|_{1,y}^{0,0,N} / {R^\prime}^N} \cdot \sqrt{\|k\|_{1,x}^{0,0,N} / {R^\prime}^N}& \text{for } R^\prime \ge R.\end{cases}\]
Choosing $N > n$, we finally get
\[\|A - A_R\|_{\mu,n}^{0,0} \le \sqrt{\|k\|_{1,y}^{0,0,N} \cdot \|k\|_{1,x}^{0,0,N}} \cdot R^{n-N}\]
finishing the proof.
\end{proof}

\section{Mapping the rough assembly map to homology}

In this section we will discuss uniformly finite homology, define the semi-norms we will use on them and prove Thoerem B. Afterwards, in Section~\ref{sec:rough_character}, we will define the character map $\chi_\ast\colon \PHCocont_\ast(\Cpol(E)) \to \Hufpol_\ast(Y)$.

\subsection{Uniformly (locally) finite homology}
\label{sec:defn_ufpol_hom}

In this section we will recall the definition of uniformly finite homology and introduce the notion of polynomial connectedness.

\begin{defn}[Bounded geometry]
\label{defn:coarsely_bounded_geometry}
Let $X$ be a metric space. It is said to have \emph{bounded geometry} if it admits a subset $Y \subset X$ with the properties
\begin{itemize}
\item there is a $c > 0$ such that $B_c(Y) = X$ and
\item for all $r > 0$ there is a $K_r > 0$ such that $\card(Y \cap B_r(x)) \le K_r$ for all $x \in X$.
\end{itemize}
Such a subset $Y$ is called a \emph{discretization of $X$}.
\qed
\end{defn}

\begin{examples}\label{ex:coarsely_bounded_geometry}
Every Riemannian manifold $M$ of bounded geometry\footnote{That is to say, the injectivity radius of $M$ is uniformly positive and the curvature tensor and all its derivatives are bounded in sup-norm.} is a metric space of bounded geometry: any maximal set $Y \subset M$ of points which are at least a fixed distance apart (i.e., there is an $\varepsilon > 0$ such that $d(x, y) \ge \varepsilon$ for all $x \not= y \in Y$) will do the job. We can get such a maximal set by invoking Zorn's lemma.

If $(X,d)$ is an arbitrary metric space that is bounded, i.e., $d(x,x^\prime) < D$ for all $x, x^\prime \in X$ and some $D$, then any finite subset of $X$ will constitute a discretization.

Let $K$ be a simplicial complex of bounded geometry\footnote{That is, the number of simplices in the link of each vertex is uniformly bounded.}. If we equip $K$ with the length metric derived from barycentric coordinates, then the set of all vertices in $K$ becomes a discretization in $K$.
\qed
\end{examples}

Let us now discuss uniformly finite homology. We will first give its definition and then discuss its functoriality under rough maps (which will be defined further below).

\begin{defn}[{\cite[Section 2]{block_weinberger_1}}]\label{defn98342d}
Let $X$ be a metric space.

$C_i^\uf(X)$ denotes the vector space of all infinite formal sums $c = \sum a_{\bar x} \bar x$ with $\bar x \in X^{i+1}$ and $a_{\bar x} \in \IC$ satisfying the following three conditions (constants depending on $c$):
\begin{enumerate}
\item\label{iu23wed} There exists $K > 0$ such that $| a_{\bar x} | \le K$ for all $\bar x \in X^{i+1}$.
\item\label{ub234ewr} For all $r > 0$ exists $K_r > 0$ with $\card \{ \bar x \in B_r(\bar y) \ | \ a_{\bar x} \not= 0 \} \le K_r$ for all $\bar y \in X^{i+1}$.
\item\label{23wefd} There is $R > 0$ such that $a_{\bar x} = 0$ if $d(\bar x, \Delta) > R$; $\Delta$ is the multidiagonal in $X^{i+1}$.
\end{enumerate}

The boundary map $\partial \colon C_i^\uf(X) \to C_{i-1}^\uf(X)$ is defined by
\[\partial (x_0, \ldots, x_i) = \sum_{j=0}^i (-1)^j (x_0, \ldots, \hat x_j, \ldots, x_i)\]
and extended by linearity to all of $C_i^\uf(X)$.

The resulting homology is the \emph{uniformly finite homology} $\Huf_\ast(X)$.
\qed
\end{defn}

Block and Weinberger proved that this homology is functorial for rough maps and invariant under rough equivalences. Let us recall these notions now:

\begin{defn}[Rough maps]\label{defn:coarse_rough_maps}
Let $f\colon X \to Y$ be a (not necessarily continuous) map. We call $f$ a \emph{rough} map, if:
\begin{itemize}
\item For all $R > 0$ there is an $S > 0$ such that we have for all $x_1, x_2 \in X$
\[d(x_1, x_2) < R \Rightarrow d(f(x_1), f(x_2)) < S.\]
\item for all $R > 0$ there is an $S > 0$ such that we have for all $x_1, x_2 \in X$
\[d(f(x_1), f(x_2)) < R \Rightarrow d(x_1, x_2) < S.\]
\end{itemize}

Two (not necessarily continuous) maps $f, g \colon X \to Y$ are called \emph{close}, if there is an $R > 0$ such that $d(f(x), g(x)) < R$ for all $x \in X$.

Two metric spaces $X$, $Y$ are \emph{roughly equivalent}, if there are rough maps $f \colon X \to Y$ and $g \colon Y \to X$ such that their composites are close to the corresponding identity maps.
\qed
\end{defn}

It is clear that if $Y \subset X$ is a discretization, then $Y$ and $X$ are roughly equivalent. So especially $\Huf_\ast(X) \cong \Huf_\ast(Y)$.



If the space has certain contractibility properties, then one can relate uniformly finite homology to a certain topological homology theory. Let us introduce now this topological homology theory and then relate it to uniformly finite homology.

\begin{defn}[Uniformly locally finite homology {\cite{engel_wrongway}}]
Let $X$ be a metric space.

A uniformly locally finite $n$-chain on $X$ is a (possibly infinite) formal sum $\sum_{\alpha \in I} a_\alpha \sigma_\alpha$ with $a_\alpha \in \IC$ and $\sigma_\alpha \colon \Delta^n \to X$ continuous for all $\alpha \in I$, where $I$ is some index set, satisfying the following conditions:
\begin{itemize}
\item $\sup_{\alpha \in I} |a_\alpha| < \infty$,
\item for every $r > 0$ exists $K_r < \infty$ such that the ball $B_r(x)$ of radius $r$ around any point $x \in X$ meets at most $K_r$ simplices $\sigma_\alpha$, and
\item the family of maps $\{\sigma_\alpha\}_{\alpha \in I}$ is equicontinuous.
\end{itemize}
We equip the chain groups with the usual boundary operator and denote the resulting homology by $\Hulf_\ast(X)$.
\qed
\end{defn}

\begin{defn}[$L^\infty$-homology]
Let $X$ be a simplicial complex of bounded geometry and equip it with the length metric derived from barycentric coordinates. We can define $L^\infty$-homology $H_\ast^\infty(X)$ by using as chains possibly infinite, formal sums of simplices of $X$ such that the coefficients of a chain have a common upper bound.
\qed
\end{defn}

We have a natural map $H_\ast^\infty(X) \to \Hulf_\ast(X)$ given by the ``identity'', i.e., an $n$-simplex of $X$ is considered as a function $\Delta^n \to X$.

\begin{lem}\label{lemh23wefwe}
We have $H_\ast^\infty(X) \xrightarrow{\cong} \Hulf_\ast(X)$.
\end{lem}

\begin{proof}
To construct a map $\Culf_n(X) \to C^\infty_n(X)$ we hit a ulf-chain as long with barycentric subdivision as the maximal diameter of a simplex in the chain is bigger or equal than a constant $C_n$. Here $C_n$ denotes the distance of the midpoint of the standard $n$-simplex $\Delta_n \subset \IR^n$ to the boundary $\partial \Delta_n$ of it (note that this distance goes to $0$ as $n \to \infty$ and that this is the metric that is used on each of the simplices of $X$).

After we made by the above procedure the simplices of a ulf-chain small enough, we can move the vertices of these simplices to the vertices of $X$. Since the simplices in our ulf-chain are small, this will provide us a ulf-chain whose simplices coincide with the simplices of $X$, i.e., it will provide us with a simplicial $L^\infty$-chain.
\end{proof}

\begin{defn}\label{defnjk2323}
Let $X$ be a metric space.

$X$ is said to be \emph{equicontinuously $k$-connected} if for every $q \le k$ it has the following property: any equicontinuous collection of maps $\{S^q \to X\}$ from the standard $q$-dimensional sphere to $X$ is equicontinuously contractible.

$X$ is said to be \emph{polynomially $k$-connected} if for every $q \le k$ it enjoys the following property: there is a polynomial $P$ such that if $S^q \to X$ is $L$-Lipschitz, then there exists a contraction of it which is $P(L)$-Lipschitz.
\qed
\end{defn}

\begin{rem}
An earlier arXiv version of the present paper used slightly weaker definitions than the ones given now in Definition~\ref{defnjk2323}. The reason for changing them was that the author noticed that the proof of Theorem~\ref{thm:bounded_below} does not work with the weaker definitions. Fortunately, all known examples of spaces satisfying the weaker definitions from the previous version on the arXiv also satisfy the stronger ones from above (see Example~\ref{exjk2323} and Lemma~\ref{lem:spaces_polynomial}).
\qed
\end{rem}

\begin{example}\label{exjk2323}
Let $M$ be a closed Riemannian manifold and equip its universal cover $X$ with the pull-back Riemannian metric.

If $\pi_i(M) = 0$ for $2 \le i \le k$, then $X$ is equicontinuously $k$-connected.
\qed
\end{example}

\begin{lem}\label{lem:spaces_polynomial}
Let $X$ be the universal cover of a connected finite simplicial complex $K$ with $\pi_1(K)$ virtually nilpotent and $\pi_i(K) = 0$ for $2 \le i \le k$.\footnote{We can drop the condition on the higher homotopy groups of $K$. In this case the statement is that $X$ will be polynomially $1$-connected.} Equip $K$ with the metric derived from barycentric coordinates and $X$ with the pull-back metric.

Then $X$ is a polynomially $k$-connected simplicial complex of bounded geometry and of polynomial volume growth (we can take $Y \subset X$ to be its set of vertices).
\end{lem}

\begin{proof}
$X$ has bounded geometry since it is a covering space of a finite complex.

Wolf proved in \cite{wolf_growth} that a finitely generated, virtually nilpotent group necessarily has polynomial growth, and the converse statement, i.e., that a finitely generated group of polynomial growth is virtually nilpotent, was proven by Gromov in \cite{gromov_polynomial}.

It remains to show that $X$ is polynomially $k$-connected. This follows from Riley's result \cite[Theorems D \& E]{riley} that virtually nilpotent groups have polynomially bounded higher-dimensional Dehn functions. Note that Tim Riley uses in the definition of Dehn functions the distance in the domain whereas we need to use the distance in the image (see the two paragraphs after the second display on Page 1337 of the cited article of Riley for the explanation of what we mean here).

Higher-order Dehn functions are also discussed in \cite[Section 5]{bridson_polynomial_dehn} and \cite{alonso_wang_pride}.
\end{proof}

A similar notion as equicontinuous $k$-connectedness, namely the notion of uniform $k$-connectedness, already appeared in John Roe's article \cite[Remark on Page 33]{roe_coarse_cohomology} and the idea used in the proof of the next lemma is from \cite[Proof of Theorem 5.28]{roe_lectures_coarse_geometry}.

We have a natural map $\Hulf_\ast(X) \to \Huf_\ast(X)$ which is given by mapping a simplex in $X$ to its tuple of vertices.

\begin{lem}\label{lemui2323}
Let $X$ be equicontinuously $k$-connected.

Then the map $\Hulf_q(X) \to \Huf_q(X)$ is an isomorphism for all $q \le k$.
\end{lem}

\begin{proof}
Let us construct a chain map $\Cuf_\ast(X) \to \Culf_\ast(X)$ for $\ast \le k+1$: given a collection of $k+2$ points in $X$ we can inductively exploit the $k$-connectedness of $X$ to get a continuous map $\Delta^{k+1} \to X$ such that the vertices of this simplex will be the $k+2$ points. Since the $k$-connectedness of $X$ is equicontinuous, we can conclude that applying this procedure to a uniformly finite chain results in a uniformly locally finite chain. This provides the promised map $\Cuf_\ast(X) \to \Culf_\ast(X)$ for $\ast \le k+1$. That it is a chain map, i.e., compatible with the boundary operator, follows from the fact that we constructed it inductively from lower to higher dimensional simplices.

Since it is a chain map it will map cycles to cycles, but it decends to homology classes only for $\ast \le k$, i.e., we get induced maps $\Huf_\ast(X) \to \Hulf_\ast(X)$ for all $\ast \le k$.

Let $\Culf_\ast(X) \to \Cuf_\ast(X)$ be the natural map. It is clear from the above construction that the composition $\Cuf_\ast(X) \to \Culf_\ast(X) \to \Cuf_\ast(X)$ is the identity (note that this composition only exists for $\ast \le k+1$).

The other composition $\Culf_\ast(X) \to \Cuf_\ast(X) \to \Culf_\ast(X)$ is in general not the identity, but it is chain homotopic to it. Since $X$ is equicontinuously $k$-connected, i.e., we can construct maps on the chain level up to degree $k+1$, the composition of the induced maps $\Hulf_\ast(X) \to \Huf_\ast(X) \to \Hulf_\ast(X)$ will be the identity for all $\ast \le k$.
\end{proof}

Note that being polynomially $k$-connected does not immediately imply that the space is also equicontinuously $k$-connected. The problem is that polynomial $k$-connectedness only controls Lipschitz maps, but not continuous one. But this is a problem which does not occur in nice spaces:

\begin{lem}\label{lemjn23cwd}
Let $X$ be a simplicial complex of bounded geometry.

If $X$ is polynomially $k$-connected, then it is equicontinuously $k$-connected.
\end{lem}

\begin{proof}
Let an equicontinuous collection of maps $\{f_j \colon S^q \to X\}_{j \in J}$ be given for $q \le k$. Fix a $\delta < C_n$, where $C_n$ is the constant appearing in the proof of Lemma~\ref{lemh23wefwe} and $n = \dim(X)$. We can approximate the whole family of maps $\{f_j \colon S^q \to X\}_{j \in J}$ up to an error of $\delta$ in each map by a family $\{f_j^\prime \colon S^q \to X\}_{j \in J}$ such that every $f_j^\prime$ is $L$-Lipschitz for a common constant $L$. Now we use polynomial $k$-connectedness to get contractions of these spheres. So it remains to show that we can homotop the family $\{f_j \colon S^q \to X\}_{j \in J}$ into the family $\{f_j^\prime \colon S^q \to X\}_{j \in J}$ equicontinuously. But this can be done via piece-wise linear homotopies of length at most $\delta$. The choice $\delta < C_n$ guarantees us that we have unique geodesics between points which are at most a distance $\delta$ apart (this is the local property that we actually need, i.e., the proof works not only for simplicial complexes but for any metric space with this kind of uniformly local uniqueness of geodesics).
\end{proof}

\begin{rem}
Let us define a space $X$ to be Lipschitz $k$-connected, if for every $q \le k$ it has the property that there is a function $F\colon \IR_{> 0} \to \IR_{> 0}$ such that if $S^q \to X$ is an $L$-Lipschitz map, then there exists a contraction of it which is $F(L)$-Lipschitz, i.e., the definition is the same as of polynomial $k$-connectedness but without the polynomial control. The above proof of Lemma~\ref{lemjn23cwd} shows that being Lipschitz $k$-connected implies being equicontinuously $k$-connected.

But the idea of the proof of Lemma~\ref{lemjn23cwd} also shows the other direction: we look at the family of all $L$-Lipschitz maps $S^q \to X$ for $q \le k$. The space, being equicontinuously $k$-connected, provides us with an equicontinuous collection of contractions. These can be now approximated up to $\delta$ by Lipschitz contractions with a common bound on their Lipschitz number. And finally, we use the local piece-wise linear homotopies to get from the original family of $L$-Lipschitz maps to the one on which the Lipschitz homotopies start. Since piece-wise linear maps have Lipschitz constant $1$, we are done.
\qed
\end{rem}

\subsection{Semi-norms on uniformly finite homology}
\label{secdefnonrms}

Let us now come to semi-norms on uniformly finite homology, which will play a crucial role in the present paper, and then prove the for this paper crucial Theorem~\ref{thm:bounded_below}, and Corollary~\ref{coriso} then proves Theorem B from the introduction.

\begin{defn}\label{defn_pol_chains}
For every $n \in \IN_0$ we define the following norm of a uniformly finite chain $c = \sum a_{\bar y} {\bar y} \in \Cuf_q(Y)$ of a uniformly discrete space $Y$:
\[\|c\|_{\infty,n} := \sup_{\bar y \in Y^{q+1}} |a_{\bar y}| \cdot \length(\bar y)^n,\]
where $\length(\bar y) = \max_{0 \le k,l \le q} d(y_k,y_l)$ for $\bar y = (y_0, \ldots, y_q)$.

We equip $\Cuf_q(Y)$ with the family of norms $(\|\largecdot\|_{\infty,n} + \|\partial \largecdot\|_{\infty,n})_{n \in \IN_0}$, denote its completion to a \Frechet space by $\Cufpol_q(Y)$ and the resulting homology by $\Hufpol_\ast(Y)$.
\qed
\end{defn}

The author used in his paper \cite{engel_BSNC} also a version of polynomially weighted norms, but for homology groups of groups. But note that there they are of a different kind: here we use an $\ell^\infty$-version, whereas in \cite{engel_BSNC} an $\ell^1$-version is used. This remark is to warn the reader not to confuse these two version, because the notion used is very similar.

\begin{thm}\label{thm:bounded_below}
Let $X$ be a polynomially $k$-connected simplicial complex of bounded geometry, let $Y \subset X$ be its set of vertices and let $Y$ have polynomial growth.

Then the chain map\footnote{This map is the composition of the map constructed in the proof of Lemma~\ref{lemui2323} with the map constructed in the proof of Lemma~\ref{lemh23wefwe}. Note that this map is actually not unique, i.e., some choices are involved (which disappear after passing to homology groups).} $\Cuf_q(Y) \to C^\infty_q(X)$ is continuous, where we equip $C^\infty_q(X)$ with the $L^\infty$-norm and $\Cuf_q(Y)$ with the above defined $\|\largecdot\|_{\infty,n}$-norm (for a concrete $n \in \IN$ which will be determined in the proof).
\end{thm}

\begin{proof}
Let $c = \sum a_{\bar y} \bar y \in \Cuf_q(Y)$ and denote by $\Delta_c \in C_q ^\infty(X)$ the simplicial $L^\infty$-chain that we get by composing the map $\Cuf_q(Y) \to \Culf_q(X)$ from the proof of Lemma~\ref{lemui2323} and the map $\Culf_q(X) \to C^\infty_q(X)$ from the proof of Lemma~\ref{lemh23wefwe}. Let $P$ be the polynomial controlling the $k$-connectedness of $X$, and let the growth of $Y$ be estimated by $\vol B_R(y) \le D \cdot R^M$ for all $y \in Y$ (we will usually just write $\vol B_R$ without mentioning $y$). We define
\[\|c\|_{R-[1]} := \sup_{\bar y \in B_{R-[1]}(\Delta_Y^{q+1})} |a_{\bar y}|,\]
where $B_{R-[1]}(\Delta_Y^{q+1}) = B_{R}(\Delta_Y^{q+1}) - \interior\big(B_{R-1}(\Delta_Y^{q+1})\big)$ for $\Delta_Y^{q+1} \subset Y^{q+1}$ the diagonal, where $B_R(\largecdot)$ denotes the closed ball of radius $R$ and $\interior(\largecdot)$ denotes the interior (i.e., $\interior(B_{R-1}(\largecdot))$ is the open ball of radius $R-1$), and then we have the estimate
\begin{equation}
\label{eqndsj3223}
|\Delta_{\bar y_0}| \le \sum_{R \in \IN} \|c\|_{R-[1]} \cdot \vol B_{P^{\circ q}(R)} \cdot \vol B_{R-[1]} \cdot (\vol B_R)^{q-1} \cdot S(R) \cdot N(R)
\end{equation}
on the coefficient of the simplex $\Delta_{\bar y_0}$ of $\Delta_c$. Here $S(R) := \lceil \log_{\sfrac{n}{n+1}} (\sfrac{C_q}{P^{\circ q}(R)}) \rceil$, where $C_q$ is the constant appearing in the proof of Lemma~\ref{lemh23wefwe}, and $N(R) := {(q+1)!}^{S(R)}$.

We arrive at this estimate in the following way: firstly, we have to estimate which of the $\bar y$ may contribute to $\Delta_{\bar y_0}$. Our construction of $\Delta_c$ is such that $\Delta_{\bar y_0}$ must lie in the $P^{\circ q}(R)$-ball of $\bar y$ for this to happen. Now we go through all $\bar y$ with maximal edge length between $R-1$ and $R$ which may constribute to $\Delta_{\bar y_0}$ by picking a point in the $P^{\circ q}(R)$-ball of $\bar y_0$ (which will be one of the vertices of a potentially contributing $\bar y$), picking a second point in the $(R-[1])$-ball of the first point (which will be our second vertex of $\bar y$ and together with the first one they will produce the maximal edge length of $\bar y$), and then picking the missing $q-1$ vertices of $\bar y$ (which must lie in the $R$-ball around the first picked point since $\bar y$ must have edge lengths smaller than $R$). This gives the first four factors in the above estimate.

Secondly, we have to know how often a single coarse simplex $\bar y$ from $c$ can contribute to $\Delta_{\bar y_0}$ after making $\bar y$ a topological simplex. This gives the last two factors in \eqref{eqndsj3223}: $S(R)$ is the number of times we have to hit the resulting topological simplex with barycentric subdivision in order to make its diameter less than $C_q$, and $N(R)$ is the number of small simplices we get by hitting it that many times with the barycentric subdivision.

By assumption $P(R)$ is polynomial and therefore $P^{\circ q}(R)$ is polynomial, and $\vol B_R$ is also polynomially bounded by assumption. Furthermore, $S(R)$ is clearly polynomially bounded, and $N(R)$ is polynomially bounded, too. So all factors but the first one of~\eqref{eqndsj3223} are polynomially bounded. Hence we get the estimate
\[\|\Delta_c\|_\infty \le \sum_{R \in \IN} \|c\|_{R-[1]} \cdot L(R),\]
where $L(R)$ is a polynomial.

We have for all $n \in \IN_0$
\begin{align*}
\|c\|_{\infty,n} & = \sup_{\bar y \in Y^{q+1}} |a_{\bar y}| \cdot \length(\bar y)^n\\
& \ge \sup_{R \in \IN} \sup_{\bar y \in B_{R-[1]}(\Delta_Y^{q+1})} |a_{\bar y}| \cdot (R-1)^n\\
& = \sup_{R \in \IN} \|c\|_{R-[1]} \cdot (R-1)^n
\end{align*}
which gives the estimate $\|c\|_{R-[1]} \le 2^n \frac{\|c\|_{\infty,n}}{R^n}$ for all $R \ge 2$ (this restriction is because we used $R-1 \ge R/2$ for $R \ge 2$). Setting $n_0 := \deg(L) + 2$ and combining everything we get
\begin{align}
\|\Delta_c\|_\infty & \le 2^{n_0} \|c\|_{\infty, n_0} \sum_{R \in \IN_{\ge 2}} R^{-2} + \|c\|_{1-[1]} L(1)\notag \\
& \le (2^{n_0} \pi^2/6 + L(1)) \|c\|_{\infty,n_0}.\label{eq:crucial_estimate}
\end{align}

Equipping $\Cuf_q(Y)$ with the norm $\|\largecdot\|_{\infty,n} + \|\partial\largecdot\|_{\infty,n}$ for $n := \deg(L) + 2$ and $C_q^\infty(X)$ with the norm $\|\largecdot\|_\infty + \|\partial\largecdot\|_\infty$, we get from the Estimate~\ref{eq:crucial_estimate} and the fact that our map $\Cuf_q(Y) \to C_q^\infty(X)$ is a chain map, that it is bounded.
\end{proof}

\begin{cor}\label{coriso}
Let $X$ be a polynomially $k$-connected simplicial complex of bounded geometry, let $Y \subset X$ be its set of vertices and let $Y$ have polynomial growth.

Then for all $q \le k$ the map $H^\infty_q(X) \to \Hufpol_q(Y)$ is a continuous isomorphism with continuous inverse, if we equip $H^\infty_q(X)$ with the $L^\infty$-seminorm and $\Hufpol_q(Y)$ with the above defined $\|\largecdot\|_{\infty,n}$-seminorm (for a concrete $n \in \IN$ which was determined in the proof of Theorem~\ref{thm:bounded_below}).
\end{cor}

\begin{proof}
The natural map $C^\infty_q(X) \to \Cufpol_q(Y)$ is continuous by definition of the norm on the latter space. The map $\Cuf_q(Y) \to C_q^\infty(X)$ is continuous by the above Theorem~\ref{thm:bounded_below} and therefore it extends to the completion $\Cufpol_q(Y)$ of its domain (note that $C^\infty_q(X)$ is a complete space, i.e., a Banach space).

We have to look at the compositions. The composition $C^\infty_q(X) \to \Cuf_q(Y) \to C^\infty_q(X)$ is not necessarily the identity: a simplex from a chain in $C^\infty_q(X)$ will be first homotoped to a possibly larger continuous simplex, then we apply several times barycentric subdivision to crush it again into smaller simplices, and lastly we move the vertices of the resulting topological simplices to the original vertices $Y$. Now note that the size of the topological simplex we get is uniformly bounded (since we start always with simplices of edge length $1$), and therefore the number of times we have to apply barycentric subdivision is uniformly bounded from above (it depends only on the dimension of the simplex in this procedure). From this we can derive that the composition  we are currently looking at, i.e., $C^\infty_q(X) \to \Cuf_q(Y) \to C^\infty_q(X)$, is bounded from above in norm. But this also shows that the chain homotopy from this composition to the identity chain map is bounded from above. These arguments also apply to the composition $C^\infty_q(X) \to \Cufpol_q(Y) \to C^\infty_q(X)$.

For the other composition $\Cuf_q(Y) \to C^\infty_q(X) \to \Cuf_q(Y)$ we can argue similarly, i.e., that the chain homotopy to the identity chain map is continuous, using the ideas from the proof of the above Theorem~\ref{thm:bounded_below}. So the chain homotopy extends to $\Cufpol_q(Y)$, which finishes this proof.
\end{proof}

\subsection{Rough character map}
\label{sec:rough_character}

Let us first quickly recall the definitions of Connes' cyclic homology theory and of the Chern--Connes character maps $K_\ast^\alg(A) \to \HC_\ast(A)$ for $A$ an algebra over $\IC$. Introductions to this topic are, e.g., Cuntz \cite{cuntz_encyc_math_sci} and Loday \cite{loday_cyclic_hom}.

\begin{defn}
The \emph{cyclic homology} $\HC_\ast(A)$ is defined as the homology of the complex
\[\ldots \longrightarrow \Cl_n(A) \stackrel{b}\longrightarrow \Cl_{n-1}(A) \stackrel{b}\longrightarrow \ldots \stackrel{b}\longrightarrow \Cl_0(A) \longrightarrow 0,\]
where $\Cl_n(A) = A^{\otimes (n+1)} / (1-\lambda) A^{\otimes (n+1)}$, $\lambda$ is the operator
\[\lambda(a_0 \otimes \cdots \otimes a_n) = (-1)^n a_n \otimes a_0 \otimes \cdots \otimes a_{n-1}\]
and $b$ is the Hochschild operator
\[b(a_0 \otimes \cdots \otimes a_n) = \sum_{j=0}^{n-1} (-1)^j a_0 \otimes \cdots \otimes a_j a_{j+1} \otimes \cdots \otimes a_n + (-1)^n a_n a_0 \otimes a_1 \cdots \otimes a_{n-1}.\]
Here we use the algebraic tensor product over $\IC$.
\qed
\end{defn}

\begin{defn}
The Chern--Connes characters $\ch_\ast \colon K^\alg_\ast(A) \to \HC_\ast (A)$ are given by the formulas
\[\ch_{2n}([e]) = \frac{(2n)!}{n!}(2\pi i)^n \trace e^{\otimes (2n+1)} \in \HC_{2n} (A)\]
for $[e] \in K^\alg_0(A)$, and for $[u] \in K^\alg_1(A)$ by
\[\ch_{1,2n+1}([u]) = \frac{(2n+1)!}{(n+1)!} (2 \pi i)^{n+1}\trace \big( (u^{-1}-1) \otimes (u-1) \big)^{\otimes (n+1)} \in \HC_{2n+1}(A).\]

Note that we have incorporated the constants appearing in the local index theorem of Connes and Moscovici into the defining formulas of the Chern--Connes characters so that at the end we do not get any constants in the statement of Theorem~\ref{thmoi2343rew}, resp., so that we get the commutativity of Diagram \eqref{eq:main_diag} not only up to constants. Compare this to J.~Roe's theorem \cite[Sections~4.5 and 4.6]{roe_coarse_cohomology}, where this constants do appear.
\qed
\end{defn}

\begin{rem}
If $A$ is a complete, locally convex algebra with jointly continuous multiplication, then we also have maps $K_\ast(A) \to \HCocont_{2n+\ast}(A)$, where for the definition of the continuous cyclic homology we have to use the completed projective tensor product in the definition of the chain groups $C_n^\lambda(A)$. See \cite{cuntz_thom} for the development of algebraic $K$-theory for such topological algebras.
\qed
\end{rem}

We are now going to define the \emph{algebraic rough character map $\HC_\ast(\uRoe(E)) \to \Huf_\ast(Y)$} for $E \to M$ a vector bundle of bounded geometry over the manifold $M$ of bounded geometry and $Y \subset M$ will be a discretization\footnote{see Definition \ref{defn:coarsely_bounded_geometry}}.

\begin{defn}\label{defnui3ed}
Let $A_0 \otimes \cdots \otimes A_n \in \uRoe(E)^{\otimes (n+1)}$. We define $\chi(A_0 \otimes \cdots \otimes A_n) \in C_n^\uf(Y)$ for $Y \subset M$ a discretization by
\[\chi(A_0 \otimes \cdots \otimes A_n)(y_0, \ldots, y_n) := \frac{1}{(n+1)!} \sum_{\sigma \in \mathfrak{S}_{n+1}} (-1)^\sigma \trace \big( A_0 y_{\sigma(0)} \cdots A_n y_{\sigma(n)} \big),\]
where $y_i$ are the projection operators on $L^2(E)$ given by characteristic functions of $V_{y_i} \subset M$, where $\{V_y\}_{y \in Y}$ is as follows:

$M$ is a manifold of bounded geometry, so it admits a compatible\footnote{\label{footnote23wesqw}Attie showed in \cite[Theorem 2.1]{attie}) that any manifold of bounded geometry $M$ may be triangulated as a simplicial complex $X$ of bounded geometry and such that the metric on $X$ derived from barycentric coordinates is bi-Lipschitz equivalent to the one on $M$.\label{footnotejkds32}} triangulation $X$ as a simplicial complex of bounded geometry. If $Y \subset X$ is the set of vertices, $Y \subset M$ is a discretization. Then we define
\begin{equation}
\label{eqjkb23wfe}
V_y := \{x \in M\colon d(x,y) \le d(x,y^\prime)\text{ for all }y^\prime \in Y\}.
\end{equation}
We can use either the distance in $M$ or the distance coming from barycentric coordinates in the above formula. Later, i.e., after passing to homology classes, this choice will not matter, since these metrics are bi-Lipschitz equivalent to each other.
\qed
\end{defn}

The operators $A_i$ in the above definition are smoothing operators and so have integral kernels $k_i \in C_b^\infty(E^\ast \boxtimes E)$ by Proposition \ref{prop:smoothing_op_kernel}. Using that for a trace class operator $T$ given by a continuous integral kernel $k(x,y)$ defined over a compact domain $N$ we have $\trace T = \int_N k(x,x) dx$, we get the formula
\begin{equation}
\label{eq:trace_integral_kernels}
\trace \big( A_0 y_{\sigma(0)} \cdots A_n y_{\sigma(n)} \big) = \int_{V_{y_{\sigma(n)}}} \hspace*{-1.9em} dx_n \cdots \int_{V_{y_{\sigma(0)}}} \hspace*{-1.9em} dx_0 \trace \big( k_0(x_n,x_0) k_1(x_0,x_1) \cdots k_n(x_{n-1},x_n) \big).
\end{equation}

\begin{lem}
The map $\chi$ is well-defined, i.e., it indeed maps into $C_\ast^\uf(Y)$.
\end{lem}

\begin{proof}
Recall Definition~\ref{defn98342d} of uniformly finite homology. Point~\ref{23wefd} is satisfied by the chain $\chi(A_0 \otimes \cdots \otimes A_n)$ because the operators $A_0, \ldots, A_n$ all have finite propagation. Point~\ref{ub234ewr} is satisfied because $Y \subset M$ is a discretization.

So it remains to show Point~\ref{iu23wed}, i.e., the uniform boundedness of the coefficients of the chain $\chi(A_0 \otimes \cdots \otimes A_n)$. Applying the Cauchy--Schwarz inequality to \eqref{eq:trace_integral_kernels} we get
\begin{align*}
& \big| \trace \big( A_0 y_{\sigma(0)} \cdots A_n y_{\sigma(n)} \big) \big|\\
& \le \dim(E) \cdot \int_{V_{y_{\sigma(n)}}} \hspace*{-1.9em} dx_n \int_{V_{y_{\sigma(0)}}} \hspace*{-1.9em} dx_0 \ \|k_0(x_n,x_0)\|^2 \cdot \underbrace{\Big\| \int_{V_{y_{\sigma(n-1)}}} \hspace*{-2.7em} dx_{n-1} \cdots \int_{V_{y_{\sigma(1)}}} \hspace*{-1.9em} dx_1 \ k_1(x_0,x_1) \cdots k_n(x_{n-1},x_n)\Big\|^2}_{\le \prod_{1 \le i \le n} \|k_i\|_\infty^2 \cdot \vol\left(V_{y_{\sigma(n-1)}} \times \cdots \times V_{y_{\sigma(1)}}\right)^2}.
\end{align*}
Using \eqref{eqkj200} we can get the following estimate for the first double integral, where $C_0 < \infty$ is the constant from \eqref{eqkj200} and $k > \dim(M) / 2$ is fixed:
\begin{align*}
\int_{V_{y_{\sigma(n)}}} \hspace*{-1.9em} dx_n \int_{V_{y_{\sigma(0)}}} \hspace*{-1.9em} dx_0 \ \|k_0(x_n,x_0)\|^2 & \le \int_{V_{y_{\sigma(0)}}} \hspace*{-1.9em} dx_0 \cdot C_0 \cdot \mu_{A_0}^{-k,0}\big(\dist(V_{y_{\sigma(n)}},V_{y_{\sigma(0)}})\big)^2\\
& \le \vol\big(V_{y_{\sigma(0)}}\big) \cdot C_0 \cdot \mu_{A_0}^{-k,0}\big(\dist(V_{y_{\sigma(n)}},V_{y_{\sigma(0)}})\big)^2.
\end{align*}
Since $M$ has bounded geometry, and due to our choice of discretization $Y \subset M$ and corresponding choice of subsets $V_y$ for $y \in Y$, we know that the volumes of the subsets $V_y$ are uniformly bounded from above. So putting all the estimates of the volumes together into a single constant $D < \infty$, we get
\begin{equation}
\label{eqerg453}
\big| \trace \big( A_0 y_{\sigma(0)} \cdots A_n y_{\sigma(n)} \big) \big| \le D \cdot C_0 \cdot \mu_{A_0}^{-k,0}\big(\dist(V_{y_{\sigma(n)}},V_{y_{\sigma(0)}})\big)^2 \cdot \prod_{1 \le i \le n} \|k_i\|_\infty^2.
\end{equation}
Note that our choice to single out the operator $A_0$ in this estimate was arbitrary. So in fact we can get the following estimate:
\[\big| \trace \big( A_0 y_{\sigma(0)} \cdots A_n y_{\sigma(n)} \big) \big| \le D \cdot \min_{0 \le j \le n}\Big\{C_j \cdot \mu_{A_j}^{-k,0}\big(\dist(V_{y_{\sigma(j-1)}},V_{y_{\sigma(j)}})\big)^2 \cdot \prod_{i \not= j} \|k_i\|_\infty^2\Big\}.\]
This finishes this proof since the dominating function $\mu_B(R)$ of an operator $B$ can be estimated from above by the operator norm of $B$ for all $R \ge 0$.
\end{proof}

\begin{lem}
The map $\chi$ descends to a chain map on $C_\bullet^\lambda(\uRoe(E))$.
\end{lem}

\begin{proof}
The total anti-symmetrization in the definition of $\chi$ ensures that it is well-defined on $C_n^\lambda(\IC_{-\infty}^\ast(E))$, i.e., that it satisfies
\[\chi(A_0 \otimes \cdots \otimes A_n) = \chi( (-1)^n A_n \otimes A_0 \otimes \cdots \otimes A_{n-1}).\]

To show that $\chi$ is compatible with the boundary operators, we first do the following computation using the substitution $(y_0, \ldots, y_{j-1}, x, y_{j}, \ldots, y_{n-1}) = (z_0, \ldots, z_n)$:
\begin{align*}
(n & + 1)! \cdot \sum_{x \in Y} \ \chi(A_0 \otimes \cdots \otimes A_n)(y_0, \ldots, y_{j-1}, x, y_{j}, \ldots, y_{n-1})\\
& = \sum_{z_j \in Y} \sum_{\sigma \in \mathfrak{S}_{n+1}} (-1)^\sigma \trace \big( A_0 z_{\sigma(0)} \cdots A_n z_{\sigma(n)} \big)\\
& = \sum_{0 \le i \le n} \sum_{\sigma \in \mathfrak{S}_{n+1}\atop \sigma(i)=j} (-1)^{\sigma} \sum_{z_j \in Y} \trace \big( A_{0} z_{\sigma(0)} \cdots A_{i} z_j A_{i+1} z_{\sigma(i+1)} \cdots A_{n} z_{\sigma(n)} \big)\\
& = \sum_{0 \le i \le n} \sum_{\sigma \in \mathfrak{S}_{n+1}\atop \sigma(i)=j} (-1)^{\sigma} \trace \big( A_{0} z_{\sigma(0)} \cdots A_{i} A_{i+1} z_{\sigma(i+1)} \cdots A_{n} z_{\sigma(n)} \big)\\
& = \sum_{0 \le i \le n} \sum_{\theta \in \mathfrak{S}_{n}} (-1)^{i+j} (-1)^{\theta} \trace \big( A_{0} y_{\theta(0)} \cdots A_{i} A_{i+1} y_{\theta(i)} \cdots A_{n} y_{\theta(n-1)} \big).
\end{align*}
So we have
\begin{align*}
\partial \big( \chi & (A_0 \otimes \cdots \otimes A_n) \big) (y_0, \ldots, y_{n-1})\\
& = \sum_{0 \le j \le n} \sum_{x \in Y} (-1)^j \chi(A_0 \otimes \cdots \otimes A_n)(y_0, \ldots, y_{j-1}, x, y_{j}, \ldots, y_{n-1})\\
& = \sum_{0 \le i \le n} \underbrace{\frac{1}{(n+1)!} \sum_{0 \le j \le n}}_{=\frac{1}{n!}} \sum_{\theta \in \mathfrak{S}_{n}} (-1)^{i} (-1)^{\theta} \trace \big( A_{0} y_{\theta(0)} \cdots A_{i} A_{i+1} y_{\theta(i)} \cdots A_{n} y_{\theta(n-1)} \big)\\
& = \sum_{0 \le i \le n} (-1)^i \chi(A_0 \otimes \cdots \otimes A_i A_{i+1} \otimes \cdots \otimes A_n)(y_0, \ldots, y_{n-1})\\
& = \chi \big( b (A_0 \otimes \cdots \otimes A_n) \big) (y_0, \ldots, y_{n-1})
\end{align*}
showing that $\chi$ is a chain map.
\end{proof}

\begin{lem}\label{lem982u3iwed}
The map $\chi\colon \uRoe(E)^{\otimes (n+1)} \to \Cuf_n(Y)$ is continuous if we equip $\uRoe(E)$ with the family of semi-norms from Remark \ref{rem:defn_seminorms_uRoe} and $\Cuf_n(Y)$ with the family of semi-norms from Definition \ref{defn_pol_chains}.
\end{lem}

\begin{proof}
This follows from \eqref{eqerg453}: the semi-norms on $\Cuf_n(Y)$ put polynomial weights onto the left-hand side of \eqref{eqerg453} depending on the distance between the points $y_0, \ldots, y_n$. But the semi-norms on $\uRoe(E)$ give us that the quantity $\mu_{A_0}^{-k,0}(\largecdot)$ on the right-hand side of the estimate \eqref{eqerg453} decays polynomially fast. Finally, in order to bound the quantities $\|\largecdot\|_\infty$ by certain operator norms occuring in the family of norms used on $\uRoe(E)$ we use Proposition~\ref{prop:smoothing_op_kernel}.
\end{proof}

By Lemma~\ref{lem982u3iwed} we get that $\chi$ continuously extends to a map $\Cpol(E)^{\hatotimes(n+1)} \to \Cufpol_n(Y)$ and therefore induces a map on homology $\chi_\ast \colon \HCocont_\ast(\Cpol(E)) \to \Hufpol_\ast(Y)$.

Let us combine the results of this subsection into the next lemma.

\begin{lem}\label{lem243tfds}
We have the following diagram:
\[\xymatrix{K^{\mathrm{alg}}_\ast(\IC_{-\infty}^\ast(E)) \ar[d]_{\ch_\ast} \ar[r] & K_\ast(\Cpol(E)) \ar[d]\\
\PHC_\ast(\IC_{-\infty}^\ast(E)) \ar[d]_{\chi_\ast} \ar[r] & \PHCocont_\ast(\Cpol(E)) \ar[d]\\
\Huf_\ast(Y) \ar[r] & \Hufpol_\ast(Y)}\]
where $Y \subset M$ is a discretization of the manifold $M$ of bounded geometry and the right vertical maps are the continuous extension of the left vertical maps.

Note that we have two cases for $\ast$: either it is $\ast = 0$ on $K$-theory and $\PHC$-theory and $\ast = \text{even}$ on homology, or $\ast = 1$ on $K$-theory and $\PHC$-theory and $\ast = \text{odd}$ on homology.
\end{lem}

\section{Index theory}
\label{sec:continuous_pairing}

In this section we will revisit coarse and rough cohomology and then use the results of the previous section to deduce that these theories pair continuously with the $K$-theory of the uniform Roe algebra in the case the manifold is polynomially contractible and of polynomial volume growth. In Section~\ref{secui23ref} we will discuss some applications, where we will also construct some interesting coarse and rough cohomology classes to pair with.

\subsection{Coarse pairings}

Let us first recall coarse cohomology and how it pairs with uniformly finite homology.

\begin{defn}[{\cite[Section 2.2]{roe_coarse_cohomology}}]\label{defn:rough_cohomology}
Let $Y$ be a uniformly discrete metric space.

We define $\CX^q(Y)$ to be the space of all functions $\phi\colon Y^{q+1} \to \IC$ satisfying the condition that for all $R > 0$
\[\supp(\phi) \cap B_R(\Delta)\]
is bounded, where $\Delta \subset Y^{q+1}$ is the diagonal.

The Alexander--Spanier coboundary operator is defined by
\begin{equation}
\label{eqjkwe99}
(\partial \phi)(y_0, \ldots, y_{q+1}) := \sum_{i=0}^{q+1} (-1)^i \phi(y_0, \ldots, \hat{y_i}, \ldots, y_{q+1})
\end{equation}
and the resulting homology is called the \emph{coarse cohomology} $\HX^\ast(Y)$ of $Y$.
\qed
\end{defn}

\begin{defn}\label{defn_pairing_uf_coarse}
Given any $\phi \in \CX^q(Y)$ and $\sigma \in \Cuf_q(Y)$ we define
\[\langle \phi, \sigma\rangle := \sum_{\bar y \in Y^{q+1}} \phi(\bar y) \cdot \sigma(\bar y).\]
This decends to a bilinear pairing $\langle \cdot, \cdot \rangle \colon \HX^q(Y) \times \Huf_q(Y) \to \IC$.
\qed
\end{defn}

Let us recall the construction of the character map $c\colon \HX^\ast(Y) \to \HcdR^\ast(M)$ from \cite[Section 2.2]{roe_coarse_cohomology} for $Y \subset M$ a discretization\footnote{recall Definition \ref{defn:coarsely_bounded_geometry}} of $M$, where $\HcdR^\ast(M)$ is compactly supported de Rham cohomology. Choose a partition of unity $\{g_{y}\}_{y \in Y}$ subordinate to the open cover $\{B_1(V_y)\}_{y \in Y}$ of $M$, where $V_y$ is as in \eqref{eqjkb23wfe} and $B_1(\largecdot)$ denotes the open ball of radius $1$ around it, and let $\phi\colon Y^{q+1} \to \IC$ be a coarse $q$-cochain. Now define a compactly supported $q$-form $c(\phi)$ by
\begin{equation}
\label{eqjew999}
c(\phi)(x) := \sum_{y_0, \ldots, y_q \in Y} \phi(y_0, \ldots, y_q) \cdot g_{y_0}(x) \cdot dg_{y_1}(x) \wedge \ldots \wedge dg_{y_q}(x).
\end{equation}
The construction descends to classes and becomes independent of the partition of unity.

So we have a pairing with $\Huf_\ast(Y)$, which gives us a pairing with $K_\ast^\alg(\uRoe(M))$ by using the character map $\chi_\ast \circ \ch_\ast\colon K_\ast^\alg(\uRoe(M)) \to \Huf_\ast(Y)$. The question is now if this pairing is continuous so that it extends to a pairing with $K_\ast(\Cpol(M)) \cong K_\ast(C_u^\ast(M))$.

\begin{thm}\label{thm:always_cont_pairing}
Let $X$ be a polynomially $k$-connected simplicial complex of bounded geometry, let $Y \subset X$ be the set of vertices of $X$, and let $Y$ have polynomial growth.

Then the pairing of $\HX^q(Y)$ with $\Huf_q(Y)$ is continuous for all $q \le k$ and therefore extends to a pairing with $\Hufpol_q(Y)$.
\end{thm}

\begin{proof}
Let us first assume that $X$ is a triangulation of a manifold $M$ of bounded geometry (in the sense of Attie, see Footnote~\ref{footnotejkds32} on Page~\pageref{footnotejkds32}).

Let $\phi \in \CX^q(Y)$ and $\sigma \in \Cuf_q(Y)$ and we need an estimate for $\langle \phi, \sigma\rangle$. Now note that the latter is the same as $\langle c(\phi), \Delta_\sigma\rangle$ for closed cochains $\phi$, where $c\colon \CX^q(Y) \to C^q_{c,\mathrm{dR}}(M)$ is the character map \eqref{eqjew999} and $\Delta_\sigma \in C_q^\infty(X)$ is the simplicial chain that we constructed in the proof of Theorem \ref{thm:bounded_below}. Now we have
\[|\langle c(\phi),\Delta_\sigma\rangle| \le C \cdot \|\Delta_\sigma\|_\infty \le C^\prime \cdot \|\sigma\|_{\infty,n},\]
where $C$ is a constant which only depends on $\phi$ (basically its corresponding operator norm), the second estimate is \eqref{eq:crucial_estimate} from the proof of Theorem \ref{thm:bounded_below}, and $n \in \IN$ is fixed (i.e., depends only on $X$ and $q$).

If $X$ is not a triangulation of a manifold, we have to use a character map $c(\largecdot)$ which maps into compactly supported cohomology of $X$ (instead of the one from above that maps into compactly supported de Rham cohomology of the manifold). Such a character map was defined by Roe \cite[Paragraph 2.11 on Page 10]{roe_coarse_cohomology}.
\end{proof}

Let us recall the index theorem associated to the above coarse pairings. The case of Dirac operators is due to J.~Roe \cite{roe_coarse_cohomology} and the extension to pseudodifferential operators is due to the author \cite{engel_indices_UPDO}. But originally this kind of index theorem goes back to Connes--Moscovici \cite{connes_moscovici}.

\begin{thm}\label{thmoi2343rew}
Let $M$ be a polynomially $k$-connected Riemannian manifold of bounded geometry and of polynomial volume growth, let $P$ be a symmetric and elliptic uniform pseudodifferential operator, and let $\phi \in \HX^q(Y)$ be a coarse cohomology class for $q \le k$.

Then we have
\[\langle \phi, (\chi_\ast \circ \ch_\ast)(\mu_\ast^u [P])\rangle = \int_M c(\phi) \wedge \ind_t (P),\]
where $\mu_\ast^u [P] \in K_\ast(C_u^\ast M)$ is the rough index class of $P$ and $\ind_t (P) \in \HbdR^\ast(M)$ denotes the topological index class of $P$.
\end{thm}

\subsection{Rough pairings}

The definition of rough cohomology and the proofs of its basic properties are all due to Mavra \cite{mavra}.

\begin{defn}[Rough cohomology, {\cite[Section 4.2]{mavra}}]
Let $Y$ be a uniformly discrete metric space of bounded geometry.

Let $R^q(Y)$ be the space of all sequences $(\phi_n)_{n \in \IN}$ of functions $Y^{q+1} \to \IC$ satisfying:
\begin{itemize}
\item For all $n \in \IN$ and $R > 0$ the set
\[\supp_R(\phi_n) := \supp(\phi_n) \cap B_R(\Delta_Y^{q+1})\]
is bounded. Here $\Delta_Y^{q+1}$ denotes the multi-diagonal in $Y^{q+1}$.
\item For every $R > 0$ we have
\begin{equation}
\label{eqjk243tre}
\| \phi_n \|_R := \sum_{\bar y \in B_R(\Delta_Y^{q+1})} |\phi_n(\bar y)| \in \ell^\infty(\IN).
\end{equation}
\end{itemize}

Define $R_c^q(Y) := \{ (\phi_n)_{n \in \IN} \in R^q(Y) \colon \text{for all } R > 0 \text{ we have } \|C^1(\phi_n)\|_R \in c_0(\IN)\}$. Here $c_0(\IN)$ denotes the sequences converging to $0$, and $C^1(\largecdot)$ is Ces\'{a}ro $C^1$-summation, i.e., $C^1(\phi_n)$ is the sequence $( \sfrac{1}{n}( \phi_1 + \cdots + \phi_n ) )_{n \in \IN}$.

The rough complex of $Y$ is defined to be $\CR^\ast(Y) := R^\ast(Y) / R^\ast_c(Y)$, where we use the Alexander--Spanier coboundary \eqref{eqjkwe99}. The cohomology of this complex is called the rough cohomology $\HR^\ast(Y)$.
\qed
\end{defn}

Rough cohomology is functorial for rough maps and is invariant under rough equivalences. Therefore we can define $\HR^\ast(X)$ for a metric space $X$ of bounded geometry as $\HR^\ast(Y)$ for any discretization $Y \subset X$.

\begin{defn}[{\cite[Section 4.2]{mavra}}]\label{defnrough9023}
Given any $\phi \in R^q(Y)$ and $\sigma \in \Cuf_q(Y)$ we define
\[\langle \phi, \sigma\rangle_\omega := {\lim}_\omega \sum_{\bar y \in Y^{q+1}} C^1(\phi_n)(\bar y) \cdot \sigma(\bar y),\]
where ${\lim}_\omega \in (\ell^\infty)^\ast$ is a choice of an ultra-limit.

This decends to a bilinear pairing $\langle \cdot, \cdot \rangle \colon \HR^q(Y) \times \Huf_q(Y) \to \IC$.
\qed
\end{defn}

We have a natural map $\HX^\ast(Y) \to \HR^\ast(Y)$ given by considering a coarse cochain $\phi$ as the constant sequence $(\phi)_{n \in \IN}$. It is clear that the pairing $\HX^q(Y) \times \Huf_q(Y) \to \IC$ factors through the pairing $\HR^q(Y) \times \Huf_q(Y) \to \IC$ via this map $\HX^\ast(Y) \to \HR^\ast(Y)$.

We want to use Formula~\eqref{eqjew999} to get a compactly supported differential form out of a function $\phi_n$. Since we assume the manifold $M$ to have bounded geometry, we can choose a partition of unity with the additional property that its derivatives are uniformly bounded. Then we get the estimate $\| c(\phi_n) \|_{L^1} \le D \| \phi_n\|_R$, where the constant $D$ depends only on the chosen partition of unity and $\|-\|_R$ is defined in \eqref{eqjk243tre}. Fixing a choice of an ultra-limit $\omega$, we can therefore define $c_\omega \colon \CR^q(Y) \to ( \Omega_{b,\mathrm{dR}}^{m-q}(M) )^\ast$ by setting
\begin{equation}
\label{eq24354tfsd}
c_\omega(\phi_n)(\alpha) := {\lim}_\omega \int_M \alpha \wedge c(C^1(\phi_n)).
\end{equation}
This descends to classes, i.e., we get a map $c_\omega \colon \HR^q(Y) \to (\HbdR^{m-q}(M))^\ast$.

\begin{thm}\label{thm9999}
Let $M$ be a polynomially $k$-connected manifold of bounded geometry, let $Y \subset M$ be a discretization of $M$, and let $Y$ have polynomial growth.

Then the pairing of $\HR^q(Y)$ with $\Huf_q(Y)$ is continuous for all $q \le k$ and therefore extends to a pairing with $\Hufpol_q(Y)$.
\end{thm}

\begin{proof}
Analogous as the proof of Theorem~\ref{thm:always_cont_pairing}.
\end{proof}

There is of course also an analogue of Theorem~\ref{thmoi2343rew}, which reads exactly the same but the formula is now
\begin{equation}
\label{eq13243reweq}
\langle \phi, (\chi_\ast \circ \ch_\ast)(\mu_\ast^u [P])\rangle_\omega = c_\omega(\phi)(\ind_t (P)).
\end{equation}

\subsection{Applications}\label{secui23ref}

In order to apply the index theorems from the previous two sections, we need interesting ways of constructing coarse, resp.~rough cohomology classes. We will give two example of how to do this in different geometric situations.

\begin{thm}
Let $M$ be a Riemannian manifold of bounded geometry. If $M$ is polynomially contractible (i.e., polynomially $k$-connected for every $k \in \IN$) and has polynomial volume growth and if $D$ is a Dirac operator over $M$, then $\mu_\ast^u [D] \not= 0 \in K_\ast(C_u^\ast M)$.
\end{thm}

\begin{proof}
Since $M^m$ is polynomially contractible, the map $c\colon \HX^\ast(M) \to \HcdR^\ast(M)$ is an isomorphism, Roe \cite[Proposition~3.33]{roe_coarse_cohomology}. So we can find a lift $\Theta \in \HX^m(M)$ of the generator $\vartheta \in \HcdR^m(M)$. Using Theorem~\ref{thmoi2343rew} we find that
\[\langle \Theta, (\chi_\ast \circ \ch_\ast)(\mu_\ast^u [D]) \rangle = \int_M \vartheta = 1,\]
since the degree-zero component of $\ind_t (D)$ is $1 \in \HbdR^0(M)$.
\end{proof}

The following corollary is immediate from the above theorem since changing the metric to a strictly quasi-isometric one (Roe \cite[End of Section 3]{roe_index_1}) induces the identity on $K_\ast(C_u^\ast M)$ and does not change the rough index class of the spin Dirac operator $\slashed D$, and since the metric having uniformly positive scalar curvature implies $\mu_\ast^u [\slashed D] = 0$.

\begin{cor}
Let $M$ be a polynomially contractible manifold of bounded geometry and of polynomial volume growth. Then $M$ does not have a metric of uniformly positive scalar curvature in its strict quasi-isometry class.
\end{cor}

Hanke--Pape--Schick \cite{hanke_pape_schick} proved the following theorem: if $M$ is a closed, connected spin manifold with $\pi_2(M) = 0$, and $N \subset M$ a codimension two submanifold with trivial normal bundle and such that the induced map $\pi_1(N) \to \pi_1(M)$ is injective, then $\alpha(N) \in K_\ast(C^\ast \pi_1(N))$ is an obstruction against the existence of a psc-metric on $M$.

We will generalize this now to higher codimensions, but we will only have the weaker obstruction $\hat{A}(N)$ instead of $\alpha(N)$. We will also have the restriction that $\pi_1(M)$ has to be virtually nilpotent due to the restrictions of our technique, i.e., we need that the universal cover of $M$ satisfies the assumptions of Theorem \ref{thm9999}.

Note that the case that the codimension of $N$ equals the dimension of $M$ is allowed in our next theorem and is nothing else but the well-known conjecture that no closed aspherical manifold admits a psc-metric.

\begin{thm}\label{thm:codim_obstruction}
Let $M$ be a closed, connected manifold with $\pi_1(M)$ virtually nilpotent and $\pi_i(M) = 0$ for $2 \le i \le q$. Assume furthermore that $N \subset M$ is a connected submanifold of codimension $q$ and with trivial normal bundle.

If $\hat A(N) \not= 0$ then $M$ does not admit a metric of positive scalar curvature.
\end{thm}

\begin{proof}
Assume that $M$ has a metric of positive scalar curvature, and assume furthermore that the inclusion $N \to M$ induces the trivial map $\pi_1(N) \to \pi_1(M)$; we will drop the latter condition later. Let $X$ be the universal cover of $M$ equipped with the pull-back metric and choose an isometric lift $\overline{N} \subset X$ of $N$. The \Poincare dual of $\overline{N}$ defines a class $P \! D[\overline{N}] \in \HcdR^q(X)$ and because $X$ is polynomially $q$-connected by Lemma~\ref{lem:spaces_polynomial} the character map $c\colon\HX^i(X) \to \HcdR^i(X)$ is by Roe \cite[Remark after Proposition 3.33]{roe_coarse_cohomology} an isomorphism for every $i \le q$. So there exists a unique lift of $P \! D[\overline{N}] \in \HcdR^q(X)$ to a class $\theta_{\overline{N}} \in \HX^q(X)$. So by Theorem \ref{thm:always_cont_pairing} the class $\theta_{\overline{N}}$ pairs continuously with $\mu_\ast^u [\slashed D]$, where $\slashed D$ denotes the spin Dirac operator of $X$. Because we assumed $M$ to have positive scalar curvature, $X$ has uniformly positive scalar curvature and therefore the index class vanishes, i.e., $\mu_\ast^u [\slashed D] = 0$. But on the other hand we have by Theorem~\ref{thmoi2343rew}
\[\langle\theta_{\overline{N}}, (\chi_\ast \circ \ch_\ast)(\mu_\ast^u [\slashed D])\rangle = \int_M c(\theta_{\overline{N}}) \wedge \hat{A}_X = \int_M P \! D[\overline{N}] \wedge \hat{A}_X = \hat{A}(\overline{N}) = \hat{A}(N),\]
where the second-to-last equality is due to the triviality of the normal bundle of $\overline{N} \subset X$. So if $\hat{A}(N) \not= 0$, then we must have $\mu_\ast^u [\slashed D] \not= 0$.

In general $\overline{N}$ will be a cover of $N$. Since the group of deck transformations of this cover $\overline{N} \to N$ is a subgroup of $\pi_1(M)$, this deck transformations form also a virtually nilpotent group and therefore especially an amenable one. Hence we can choose a \Folner sequence for $\overline{N}$ compatible with the covering homomorphism, Roe \cite[Proposition~6.6]{roe_index_1}, and interpret their \Poincare duals as a sequence of compactly supported de Rham cohomology classes of degree $q$ such that pairing with bounded de Rham cohomology and taking an ultra-limit ${\lim}_\omega$ gives us a functional which we will denote by $P \! D_\omega[\overline{N}] \in (\HbdR^{m-q}(M))^\ast$. Using again that $X$ is polynomially $q$-connected we can find a rough cohomology class $\theta_{\overline{N}} \in \HR^q(X)$ such that $c_\omega(\theta_{\overline{N}}) = P \! D_\omega[\overline{N}]$, see \eqref{eq24354tfsd}. By Theorem~\ref{thm9999} and Equation~\eqref{eq13243reweq} we then again conclude as in the above special case (which was that $\pi_1(N) \to \pi_1(M)$ is trivial) that $\mu_\ast^u [\slashed D] \not= 0$ if $\hat{A}(N) \not= 0$.
\end{proof}

\section{Homology of the uniform Roe algebra}
\label{seciu234e}

In this section we will show that the character map $\chi_\ast\colon \PHCocont_\ast(\Cpol(M)) \to \Hufpol_\ast(Y)$ is an isomorphism if $M$ has polynomial volume growth, where we are using the trivial line bundle to define the algebra $\Cpol(M)$.

In \cite[Section~2]{yu_cyclic} arguments are given for the map $\PHC_\ast(B_M^u) \to \Huf_\ast(Y)$ being an isomorphism, where $B_M^u$ denotes an uncompleted version of the uniform Roe algebra consisting of finite propagation operators having a uniformly bounded smooth kernel. He also claims an isomorphism $\PHC_\ast(B_M) \cong \HX_\ast(Y)$, where $B_M$ denotes an uncompleted version of the Roe algebra consisting of locally traceable, bounded operators on $L^2(M)$ having finite propagation and $\HX_\ast(Y)$ is the coarse homology of $Y$.

So one might think that all we have to do in this section to prove the isomorphism $\PHCocont_\ast(\Cpol(M)) \cong \Hufpol_\ast(Y)$ would be to show that all the maps that G.~Yu constructs in his arguments are continuous and therefore everything extends to the completions. Unfortunately, the arguments are incomplete and one can actually find counter-examples to the results as stated, so that in our proof in this section we will additionally have to incorporate corrections to the arguments and also add the missing step. Let us explain the objections to the argument in more detail now.\footnote{That the arguments in \cite[Section~2]{yu_cyclic} are insufficient, resp.~not even quite correct, was noticed by Ulrich Bunke and Luigi Caputi, who told it to the author.}

The main problem is that G.~Yu works purely in the algebraic setting (i.e., algebraic cyclic homology is used and no topology is put onto the algebras $B_M^u$ and $B_M$), but this leads to all the following problems. First of all, note that the claim $\PHC_\ast(B_M) \cong \HX_\ast(Y)$ is wrong in general. We can see this by taking $M$ to be a compact manifold: then the claimed isomorphism is that $\PHC_\ast(\Tr(H)) \cong \IC$ for $H$ a separable, infinite-dimensional Hilbert space. But this is only true is we put the trace norm on $\Tr(H)$ and use continuous cyclic homology (see, e.g., Cuntz \cite[Proposition 3.5]{cuntz_encyc_math_sci}). But in the purely algebraic setting $\HC_0(\Tr(H))$ is highly non-trivial (Dykema--Figiel--Weiss--Wodzicki \cite{DFWW}) and the higher cyclic homology groups are unknown. A similar problem occurs in the claim that $\PHC_\ast(B_M^u) \cong \Huf_\ast(Y)$.

Let us explain the missing step in the proposed proof of the above claims and why the author thinks that this missing step can not be carried out in the algebraic setting. In the argument algebraic tensor products $A_0 \otimes \cdots \otimes A_n$ of operators from $B_M^u$ are identified with smooth functions on $M^{2n+2}$ by using kernel representations. This construction works well in one direction, i.e., associating an algebraic tensor product of operators a smooth function on $M^{2n+2}$, but the other direction is completely unclear in the algebraic situation: given a function on $M^{2n+2}$, it seems to the author in general not possible to determine if it comes from a linear combination of algebraic tensor products of kernel operators. In general one needs to pass to a completion of the algebraic tensor product to get the needed surjectivity (see Section~\ref{secu2343rew} for how we solve the problem in this paper). This lifting of functions on $M^{2n+2}$ to algebraic tensor products of kernel operators is the crucial missing step and where the author thinks that it is not possible to accomplish it in general and still stay in the algebraic setting.

The above discussed problems disappear by passing to the topological world, and this is the reason why we are able to carry them out here in this paper. But by passing to completions we loose the finite propagation of the operators and hence immediately end up struggling with Novikov-type problems (which is the reason why we restrict ourselves only to the case of polynomial volume growth). Not loosing finite propagation is probably also exactly the reason why G.~Yu wanted to stay purely in the algebraic world.

Recall Definition~\ref{defnrtzu} of the space of kernels $\WDpol(E \boxtimes E^\ast)$. We write $\WDpol(M \times M)$ in the case of the trivial bundle $E = \IC$. Recall that in Theorem~\ref{thmjnksd9023} we have shown that $\Cpol(M) \xrightarrow{\cong} \WDpol(M \times M)$ if $M$ has polynomial volume growth. Therefore the claimed isomorphism $\PHCocont_\ast(\Cpol(M)) \to \Hufpol_\ast(Y)$ will follow from the next theorem.

\begin{thm}\label{thmi34werd}
Let $M$ be a manifold of bounded geometry with polynomial volume growth, and let $Y \subset M$ be a discretization.

Then the character map $\chi$ from Section~\ref{sec:rough_character} induces an isomorphism
\[\chi_\ast \colon \PHCocont_\ast(\WDpol(M \times M)) \xrightarrow{\cong} \Hufpol_\ast(Y),\]
where $\ast$ is either $\ast = 0$ on $\PHC$-theory and $\ast = \text{even}$ on homology, or it is $\ast = 1$ on $\PHC$-theory and $\ast = \text{odd}$ on homology.
\end{thm}

The proof will occupy the whole of this Section~\ref{seciu234e} and will mainly consist of constructing various maps in order to construct at the end a map which will be an inverse to $\chi$ up to chain homotopy.

We will consider the following factorization (up to the total anti-symmetrization) of the character map $\chi$. We will explain the occuring maps further below. The arrows for the maps $\sigma$ and $\gamma$ are dashed since these maps will not be defined on all of $C_b^\infty(M^{n+1})$, but only on a certain subspace of it (on those functions which are supported in a uniform neighbourhood of the diagonal).
\begin{equation}
\label{eq2w43er2}
\xymatrix{
\WDpol(M \times M)^{\hatotimes (n+1)} \ar[r]^-{\alpha} & C_b^\infty((M\times M)^{n+1}) \ar[r]^-{\beta} & C_b^\infty(M^{n+1}) \ar@{-->}[r]^-{\gamma} \ar@/^1.5pc/[l]_-{\vartheta} \ar@/_2.0pc/@{-->}[ll]^-{\sigma} & \Cuf_n(Y) \ar@/^1.5pc/[l]_-{\omega}
}
\end{equation}

\subsection{The map \texorpdfstring{$\alpha$}{alpha}}

The map $\alpha$ is induced from the map
\begin{align}
\WDpol(M \times M)^{\otimes_\alg (n+1)} & \to C_b^\infty((M\times M)^{n+1}),\label{eq34re2398fd}\\
k_0 \otimes \cdots \otimes k_n & \mapsto \big( ( x_0, x_0^\prime, \ldots, x_n, x_n^\prime ) \mapsto k_0(x_0, x_0^\prime) \cdots k_n(x_n, x_n^\prime) \big)\notag
\end{align}
which is well-defined by the Sobolev embedding theorem if $M$ has bounded geometry.

Recall the definition of the projective tensor product: given normed spaces $E$ and $F$, the projective tensor product norm on $E \otimes_\alg F$ is given by
\begin{equation}
\label{eqjk34r9ij8efd}
\|u\|_{E \hatotimes F} = \inf \{ \sum \|x_i\|_E \|y_i\|_F\},
\end{equation}
where the infimum ranges over all the representations $u = \sum_i x_i \otimes y_i$. If the topologies on $E$ and $F$ are defined by families of semi-norms $(p_i)_{i \in I}$ and $(q_j)_{j \in J}$, respectively, then the projective tensor product topology on $E \otimes_\alg F$ is defined by using the family of semi-norms $(p_i \otimes q_j)_{(i,j) \in I \times J}$, where $p_i \otimes q_j$ is defined by \eqref{eqjk34r9ij8efd}.

By \eqref{eqo24ew} we see that \eqref{eq34re2398fd} is continuous with respect to the projective tensor product topology and hence extends continuously to $\WDpol(M \times M)^{\hatotimes (n+1)}$. This continuous extension is the map $\alpha$.

\begin{lem}\label{lemui2983e}
The maps \eqref{eq34re2398fd} and $\alpha$ are injective.
\end{lem}

\begin{proof}
Since $\WDpol(M \times M)$ is a Hausdorff space, by Treves \cite[Proposition~43.3]{treves} the projective tensor product topology on $\WDpol(M \times M)^{\otimes_\alg(n+1)}$ is also Hausdorff. So
\begin{equation*}
\label{eqiu43er}
\WDpol(M \times M)^{\otimes_\alg(n+1)} \to \WDpol(M \times M)^{\hatotimes(n+1)}
\end{equation*}
is injective\footnote{If the projective tensor product topology on $\WDpol(M \times M)^{\otimes_\alg(n+1)}$ would not be Hausdorff, we would have taken the quotient by the closure of $0$ in the definition of $\WDpol(M \times M)^{\hatotimes(n+1)}$ probably forcing the map \eqref{eqiu43er} to not be injective anymore.} and therefore injectivity of \eqref{eq34re2398fd} follows from injectivity of $\alpha$.

Let us consider the composition
\begin{equation*}
\label{eq453erfd}
\WDpol(M \times M)^{\hatotimes(n+1)} \xrightarrow{\alpha_0} C_b(M \times M)^{\hatotimes(n+1)} \xrightarrow{\alpha_1} C_b((M \times M)^{n+1})
\end{equation*}
and show that both of these maps are injective. Since $\alpha$ is the composition $\alpha_1 \circ \alpha_0$, we can then conclude that $\alpha$ is also injective.

Since $\WDpol(M \times M) \to C_b(M \times M)$ is injective, we get by Treves~\cite[Exercise~43.2]{treves} that the map $\alpha_0$ is also injective.

To prove injectivity of the map $\alpha_1$ we consider the following:
\begin{align*}
C_b(M \times M)^{\hatotimes(n+1)} & \hookrightarrow C_b(M \times M)^{\otimes_{\mathrm{min}}(n+1)} \cong C(\beta_{M \times M})^{\otimes_{\mathrm{min}}(n+1)} \cong C((\beta_{M \times M})^{n+1})\\
& \hookrightarrow C(\beta_{(M \times M)^{n+1}}) \cong C_b((M \times M)^{n+1}),
\end{align*}
where $\beta_{\largecdot}$ denotes the Stone--\v{C}ech compactification and $\otimes_{\mathrm{min}}$ the minimal $C^\ast$-algebra tensor product. We have to justify the first inclusion (the second one comes from the surjection $\beta_{X \times Y} \twoheadrightarrow \beta_X \times \beta_Y$): by Kumar--Sinclair \cite[Theorem 6.1]{kumar_sinclair} the projective tensor product $\hatotimes$ is equivalent to the Haagerup tensor product $\otimes_h$ for infinite-dimensional sub-homogeneous $C^\ast$-algebras (the special case of commutative $C^\ast$-algebras was originally proven by Grothendieck \cite{grothendieck} --- note that the Haagerup tensor product is not a tensor product of $C^\ast$-algebras). Now we use a result of Blecher \cite[Proposition~4.2.1]{blecher} stating that we always have an injection $A \otimes_h B \hookrightarrow A \otimes_{\mathrm{min}} B$ for any $C^\ast$-algebras $A,B$.
\end{proof}

\subsection{The maps \texorpdfstring{$\beta$}{beta} and \texorpdfstring{$\vartheta$}{theta}}

The map $\beta\colon C_b^\infty((M\times M)^{n+1}) \to C_b^\infty(M^{n+1})$ in \eqref{eq2w43er2} is defined in the following way:
\begin{align*}
\beta(f)(x_0, \ldots, x_n) := f(x_n, x_0, x_0, x_1, x_1, \ldots, x_{n-1}, x_{n-1}, x_n).
\end{align*}

To define $\vartheta$ we choose a smooth partition of unity $(g_i)_{i \in I}$ on $M$ such that the derivatives of these functions are uniformly bounded in $i$ and the corresponding covering of $M$ by their supports has finite multiplicity (later we will also want that their supports have uniformly bounded diameters). We then define $\vartheta\colon C_b^\infty(M^{n+1}) \to C_b^\infty((M\times M)^{n+1})$ by
\begin{align*}
\vartheta(& h) (x_0, x_0^\prime, x_1, x_1^\prime, \ldots, x_n, x_n^\prime)\\
& := \sum_{i_0, \ldots, i_n \in I} \sqrt{g_{i_n}(x_0) g_{i_0}(x_0^\prime) g_{i_0}(x_1) g_{i_1}(x_1^\prime) g_{i_1}(x_2) \cdots g_{i_{n-1}}(x_n) g_{i_n}(x_n^\prime)} \cdot h(x_0^\prime, x_1^\prime, \ldots, x_n^\prime).
\end{align*}
This map $\vartheta$ is a continuous right-inverse of the map $\beta$, i.e., we have $\beta \circ \vartheta = \id$.

\subsection{The maps \texorpdfstring{$\gamma$}{gamma} and \texorpdfstring{$\omega$}{omega}}

The map $\gamma\colon C_b^\infty(M^{n+1}) \dashrightarrow \Cuf_n(Y)$ is defined by
\[\gamma(f)(y_0, \ldots, y_n) := \int_{V_{y_n}} \hspace*{-0.7em} dx_n \cdots \int_{V_{y_0}} \hspace*{-0.7em} dx_0 \ f(x_0, \ldots, x_n),\]
where the subsets $V_y \subset M$ are defined in \eqref{eqjkb23wfe}. The arrow is dashed since the domain of $\gamma$ consists only of all the functions which are supported in a uniform neighbourhood of the diagonal. The map $\gamma$ is continuous if we put the norm $\|-\|_{\infty,0}$ from Definition~\ref{defn_pol_chains} on $\Cuf_n(Y)$. If we put any other norm $\|-\|_{\infty,n}$ on $\Cuf_n(Y)$ for $n \ge 1$, then $\gamma$ will not be continuous (provided $M$ is non-compact).

To define $\omega\colon \Cuf_n(Y) \to C_b^\infty(M^{n+1})$ we choose smooth positive functions $(\phi_y)_{y \in Y}$ with the following properties: $\supp \phi_y \subset V_y$ for $V_y$ as in \eqref{eqjkb23wfe}, $\int \phi_y = 1$, and all their derivatives are bounded uniformly in $y$. Given now a uf-chain $c = \sum_{\bar y \in Y^{n+1}} a_{\bar y} {\bar y} \in \Cuf_n(Y)$ we define a function on $M^{n+1}$ by
\[(x_0, \ldots, x_n) \mapsto \sum_{\bar y \in Y^{n+1}} a_{\bar y} \cdot \phi_{y_1}(x_0) \phi_{y_0}(x_1) \cdots \phi_{y_n}(x_n),\]
where we have written $\bar y = (y_0, \ldots, y_n)$. Since the subsets $V_y$ are mutually disjoint from each other (up to sets of measure zero), we get that $\omega$ is a right-inverse of $\gamma$, i.e., we have $\gamma \circ \omega = \id$. The map $\omega$ is continuous if we put on $\Cuf_n(Y)$ the \Frechet topology from Definition~\ref{defn_pol_chains} (we actually just have to put only the norm $\|-\|_{\infty,0}$ on $\Cuf_n(Y)$ to get this continuity).

\subsection{The map \texorpdfstring{$\sigma$}{sigma}}
\label{secu2343rew}

Note that the map $\alpha$ is not surjective. It is likely not even surjective onto the image of $\vartheta \circ \omega$ (in the case $n \ge 1$), which is contained in the space of all functions supported in a uniform neighbourhood of the diagonal in $(M\times M)^{n+1}$ (provided we have chosen the partitions of unity in the definition of $\vartheta$ such that they have uniformly bounded diameters). But we will show that the composition $\beta \circ \alpha$ is surjective onto the subspace of $C_b^\infty(M^{n+1})$ consisting of all functions supported in a uniform neighbourhood of the diagonal in $M^{n+1}$ (this subspace contains the image of $\omega$). We will show this surjectivity by constructing the map $\sigma$ which will be defined on this subspace (i.e., $\sigma$ will not be defined on all of $C_b^\infty(M^{n+1})$ which is the reason why its arrow in \eqref{eq2w43er2} is dotted).

Fix an $R > 0$. We define $\sigma$ now on the subspace of $C_b^\infty(M^{n+1})$ consisting of functions which are supported in the $R$-neighbourhood of the diagonal. Let $g \in C_b^\infty(M^{n+1})$ be such a function. We first apply $\vartheta$ to $g$ in order to get a function $f := \vartheta(g)$ on $C_b^\infty((M\times M)^{n+1})$ which is supported on a uniform neighbourhood of the diagonal (we assume that the partition of unity we picked for the definition of $\vartheta$ consists of functions whose supports have uniformly bounded diameters --- by enlarging $R$ if needed, we can without loss of generality assume that these diameters are bounded from above by $R$).

Since $M$ has bounded geometry we can find a uniformly locally finite covering $\{U_i\}_{i \in I}$ of $M$ by open subsets of uniformly bounded diameters such that the Lebesgue number of this covering is at least $R$. The choice of Lebesgue number ensures that the support of $f$ is contained in $\bigcup_{i \in I} (U_i \times U_i) \times \cdots \times (U_i \times U_i)$.

Note that $\{V_i := (U_i \times U_i) \times \cdots \times (U_i \times U_i)\}_{i \in I}$ is a uniformly locally finite covering of the $R$-neighbourhood of the diagonal of the manifold $(M\times M)^{n+1}$, which has bounded geometry. Therefore we can find a subordinate partition of unity $\{h_i\}$ with uniformly bounded derivatives.

On each $V_i$ we can find $k_i \in \WDpol(U_i \times U_i)^{\hatotimes(n+1)}$ with $\alpha(k_i) = h_i f$ (this equality only holds on $V_i$) since $U_i$ is compact.\footnote{This is the step which crucially needs that we work in the completion of the projective tensor product, i.e., $k_i$ will in general not live in the algebraic tensor product.}
The expression $k := \sum k_i$ (we will explain this in more detail in the following paragraph) defines an element from $\WDpol(M \times M)^{\hatotimes(n+1)}$ since the sets $U_i$ have uniformly bounded diameters and the norms of the functions $k_i$ are uniformly bounded. Now this $k$ is the sought preimage under $\alpha$ of the function $f = \vartheta(g)$, where the equality only holds on the $R$-neighbourhood of the diagonal of $(M \times M)^{n+1}$. Since the support of $g$ is contained in the $R$-neighbourhood of the diagonal in $M^{n+1}$, we have $(\beta \circ \alpha)(k) = g$.

What we mean in the previous paragraph by $\sum k_i$ is the following: suppose that we have $k_i = k_{i,0} \otimes \cdots \otimes k_{i,n}$ for each $i$ (i.e., for simplicity we assume here that each $k_i$ is an elementary tensor). Then we mean by $\sum k_i$ the expression
\[\sum k_{i,0} \otimes \cdots \otimes \sum k_{i,n}.\]
Note that cross terms like $k_{i,0} \otimes k_{i^\prime,1} \otimes \cdots$ are responsible for the equality $f = \vartheta(g)$ to hold only on the $R$-neighbourhood of the diagonal of $(M \times M)^{n+1}$, i.e., they are the reason why this construction here does not provide a right-inverse to $\alpha$, but only to $\beta \circ \alpha$ on the subspace of functions supported on a uniform neighbourhood of the diagonal. We assumed here that $k_i$ is an elementary tensor. In general, up to an $\varepsilon$, the element $k_i$ will be a finite sum of elementary tensors. Here now we need bounded geometry of $M$ in order to have a bound (independent of $i$) on the number of summands.

Now if we change our fixed value of $R$, the choice of covering $\{U_i\}_{i \in I}$ or choice of partition of unity $\{h_i\}$, get a priori different lifts $k$ and $k^\prime$. But one can see that we have $\alpha(k) = \alpha(k^\prime)$ and therefore, since $\alpha$ is injective by Lemma~\ref{lemui2983e}, we get $k=k^\prime$. Therefore we can define $\sigma$ on the whole subspace of all functions in $C_b^\infty(M^{n+1})$ which are supported in some uniform neighbourhood of the diagonal.

The composition $\sigma \circ \omega$ is continuous if the manifold $M$ has polynomial volume growth. Recall from Definition~\ref{defnrtzu} of the norms on $\WDpol(M \times M)$ that they are a supremum over polynomially weighted $L^1$-norms, whereas the norms on $C_n^\uf(Y)$ are by Definition~\ref{defn_pol_chains} a supremum over polynomially weighted $L^\infty$-norms --- so we necessarily need polynomial growth of $M$ to conclude continuity of $\sigma \circ \omega$ (the norms on $\WDpol(M \times M)$ also take the derivatives of the kernels into account, but this is no problem for the continuity of $\sigma \circ \omega$ since in the definitions of both $\omega$ and $\sigma$ we are choosing partitions of unity with uniformly bounded derivatives).

\subsection{Proof of Theorem~\ref{thmi34werd}}

Let us prove Theorem~\ref{thmi34werd}. As usual, it suffices to prove that we have isomorphisms on Hochschild homology (see, e.g., Loday~\cite[Corollary~2.2.3]{loday_cyclic_hom}).

If we restrict the domain of $\alpha$ to elementary tensors $k_0 \otimes \cdots \otimes k_n$, where each kernel function $k_i$ is supported in a uniform neighbourhood of the diagonal, then the composition $\gamma \circ \beta \circ \alpha$ is the character map $\chi$ (up to an identification of the domain using Theorem~\ref{thmjnksd9023}). From Lemma~\ref{lem982u3iwed} we know that the map $\chi$ is continuous and therefore extends to a map $\chi\colon \WDpol(M \times M)^{\hatotimes (n+1)} \to \Cufpol_n(Y)$.

We already know that the composition $\sigma \circ \omega$ is continuous on $\Cuf_n(Y)$ if $M$ has polynomial volume growth, and therefore it extends to a map on $\Cufpol_n(Y)$ which we will also denote by $\sigma \circ \omega$. It is quickly verified that this is a chain map. Since by construction $\omega$ is a right-inverse to $\gamma$ and $\sigma$ is a right-inverse to $\beta \circ \alpha$, we conclude that the composition $\chi \circ (\sigma \circ \omega)$ is the identity on $\Cufpol_n(Y)$.

We want to prove that $\sigma \circ \omega$ is also a left-inverse to $\chi$, but only up to chain homotopy this time. This would finish the proof of Theorem~\ref{thmi34werd}. Now G.~Yu already constructed a chain homotopy $D$ in his proposed proof \cite[Page~450]{yu_cyclic} and fortunately for us, the formulas also make sense if interpreted as if defined on $\WDpol(M \times M)^{\hatotimes (n+1)}$ (we have to regard the occuring products $\psi_{\gamma_i}(x_i^\prime) \psi_{\gamma_i}(x_{i+1})$ as the single kernel function $\psi_{\gamma_i} \otimes \psi_{\gamma_i}$). So basically we are defining our chain homotopy as the composition of $\alpha$ with G.~Yu's chain homotopy and then we lift the result back to $\WDpol(M \times M)^{\hatotimes (n+1)}$. We see that this chain homotopy is continuous and therefore extends to all of $\WDpol(M \times M)^{\hatotimes (n+1)}$ (to be very concrete, our recipe for the chain homotopy a priori only works on the algebraic tensor product of finite propagation kernels). Now the fact that this chain homotopy does the job follows from the corresponding computation of G.~Yu together with the map $\alpha$ being injective (Lemma~\ref{lemui2983e}).

\section{Final remarks and open questions}\label{sec:final_remarks}

Let us collect in this last section some open questions arising out of this paper and which the author thinks are worth persuing.

\subsection{Homological rough Novikov conjecture}

We have shown in Section~\ref{secdefnonrms} that $\Huf_\ast(Y) \cong \Hufpol_\ast(Y)$ if $Y \subset X$ is the vertex set of a simplicial complex $X$ of bounded geometry and of polynomial growth and if we assume that $X$ is polynomially contractible. We needed the polynomial contractibility, because in our proof we compared with $H_\ast^\infty(X)$ which sees the topology of $X$. But both $\Huf_\ast(Y)$ and $\Hufpol_\ast(Y)$ are quasi-isometry invariants of $Y$ and the question whether $\Huf_\ast(Y) \to \Hufpol_\ast(Y)$ is an isomorphism makes sense without mentioning any topology at all. Therefore we would actually expect that we should not need the assumption that $X$ is polynomially contractible (and we should actually not even need to assume that there is such a space $X$ at all in the background).

\begin{question}\label{quesjkn243re}
Let $Y$ be uniformly discrete metric space of bounded geometry and let it have polynomial volume growth. Is $\Huf_\ast(Y) \to \Hufpol_\ast(Y)$ an isomorphism?
\end{question}

Another approach to investigate polynomially decaying uf-homology is to show a result analogous to \cite[Lemma 9.5]{roe_index_coarse}:

\begin{question}
Is the functor $Y \mapsto \Hufpol_\ast(\mathcal{O}Y)$ a reduced homology theory on the category of finite polyhedra? ($\mathcal{O}Y$ denotes the Euclidean cone over $Y$.)
\end{question}

\subsection{Sparse index classes}

Theorem A states that under certain assumptions on $M$ coarse and rough cohomology pair continuously with the $K$-theory of the uniform Roe algebra. This pairing factors via the character map $\chi_\ast$ through $\Hufpol_\ast(Y)$ for $Y \subset M$ a discretization. But $\Hufpol_\ast(Y)$ is usually a non-Hausdorff space which means that, due to the continuity of the pairings, elements in the closure $\overline{\{0\}}\subset \Hufpol_\ast(Y)$ can not be detected by these pairings. Combining Theorem F with Theorem H we get the isomorphism $K_\ast(C_u^\ast(M)) \barotimes \IC \cong \Hufpol_\ast(Y)$ and therefore we know that there are elements in $K_\ast(C_u^\ast(M)) \barotimes \IC$ not detectable by coarse or rough index pairings. The question is now whether we can find such elements already in $K_\ast(C_u^\ast(M)) \otimes \IC$ or if they arise only through the procedure of taking the completion in the tensor product, and if there are such elements, how to characterize them.

\begin{question}
Are there elements in $K_\ast(C_u^\ast(M)) \otimes \IC$ which are mapped to elements in the closure of $\{0\}$ in $\Hufpol_\ast(Y)$?

If yes, can we characterize these elements and what is their index theoretic value?
\end{question}

\subsection{Assembly map on the level of cyclic (co-)homology}
\label{subsec234ter}

The left vertical map in Diagram \eqref{eq:main_diag}, i.e., the uniform homological Chern character, was defined by the author \cite[Section 5]{engel_indices_UPDO} by factoring it through the continuous periodic cyclic cohomology $\HPucont^\ast(W^{\infty,1}(M))$ of $W^{\infty,1}(M)$ which denotes the Sobolev space on $M$ of infinite order and $L^1$-integrability. So the question whether there is a rough assembly map on the level of cyclic (co-)homology immediately arises.

\begin{question}
Is there an assembly map $\HPucont^\ast(W^{\infty,1}(M)) \to \HP_\ast(\IC^\ast_{-\infty}(M))$?

If $M$ is equicontinuously contractible, is this map an isomorphism?
\end{question}

The next question is then if the analogues of our results from Section~\ref{secdefnonrms} are true:

\begin{question}
If $M$ is a polynomially contractible manifold of bounded geometry and of polynomial growth, is then $\HPucont^\ast(W^{\infty,1}(M)) \to \PHCocont_\ast(\Cpol(M))$ an isomorphism?

If $M$ is of bounded geometry and of polynomial volume growth, is then the completion map $\PHC_\ast(\IC_{-\infty}^\ast(M)) \to \PHCocont_\ast(\Cpol(M))$ an isomorphism?
\end{question}

One has to be a bit cautious with the above questions, since we stated them using the \emph{algebraic} periodic cyclic homology of $\IC_{-\infty}^\ast (M)$. It might turn out that this is maybe not appropriate due to reasons explained at the beginning of Section~\ref{seciu234e}.

\subsection{Rough assembly map}

In the previous subsection we asked about an assembly map on the level of cyclic homology in order to have a more complete picture of Diagram \eqref{eq:main_diag}. But in this diagram is another deficiency: we expect that the rough assembly map $\mu_\ast^u$ factors through $K_\ast^\alg(\uRoe(E))$.

\begin{question}
Does the rough assembly map factor through $K_\ast^\alg(\uRoe(E))$?

If yes, is the resulting algebraic rough assembly map $K_\ast^u(M) \to K_\ast^\alg(\uRoe(E))$ always an isomorphism when $M$ is equicontinuously contractible?
\end{question}

Theorem H states the isomorphism $K_\ast(C_u^\ast(M)) \barotimes \IC \xrightarrow{\cong} \PHCocont_\ast(\Cpol(M))$ if $M$ is polynomially contractible and of polynomial volume growth. But this result should be independent of the topological assumption on $M$ that it is polynomially contractible; compare this to Theorem F and to Question~\ref{quesjkn243re}.

\begin{question}
Let $M$ be a manifold of bounded geometry and of polynomial growth.

Do we have an isomorphism $\ch_\ast\colon K_\ast(C_u^\ast(M)) \barotimes \IC \xrightarrow{\cong} \PHCocont_\ast(\Cpol(M))$?
\end{question}

\bibliography{./Bibliography_Rough_index_theory_polynomial}
\bibliographystyle{amsalpha}

\end{document}